\documentclass[final,3p,times]{elsarticle}

\usepackage[colorlinks=true]{hyperref}
\usepackage{amsfonts}
\usepackage{mathrsfs}
\usepackage{amsthm,amsmath,amssymb,mathrsfs,amsfonts,graphicx,graphics,latexsym,exscale,cmmib57,dsfont,bbm,amscd,ulem,curves}
\usepackage{color,soul}
\newtheorem*{sthm}{\ref{sec0.1}(i)\; Question}
\newtheorem*{sthm-b}{\ref{sec0.1}(ii)\; Question}
\newtheorem*{thm*}{Theorem}
\newtheorem{thm}{Theorem}[section]
\newtheorem*{0.13}{Theorem~\ref{0.13}}
\newtheorem*{0.14}{Theorem~\ref{0.14}}
\newtheorem*{0.17}{Theorem~\ref{0.17}}

\newtheorem{lem}[thm]{Lemma}

\newtheorem{cor}[thm]{Corollary}

\newtheorem*{lem5.8}{Lemma~\ref{lem2.14}$^\prime$}

\theoremstyle{definition}
\newtheorem{exa}[thm]{Example}
\newtheorem{exas}[thm]{Examples}

\newtheorem{que}[thm]{Question}

\newtheorem*{note}{Note}

\newtheorem{rem}[thm]{Remark}
\newtheorem{sn}[thm]{Definition}

\newtheorem*{sdt}{Standing terminology}
\newtheorem*{sht}{Standing hypothesis and terminoloy}
\newtheorem*{sh}{Standing hypothesis}
\newtheorem{term}[thm]{Terminoloy}

\numberwithin{equation}{section}
\journal{JMAA August 1, 2018}

\begin{document}

\begin{frontmatter}
\title{Almost automorphy of surjective semiflows on compact Hausdorff spaces}
\author{Xiongping Dai}
\ead{xpdai@nju.edu.cn}
\address{Department of Mathematics, Nanjing University, Nanjing 210093, People's Republic of China}

\begin{abstract}
Let $(T,X)$ with phase mapping $(t,x)\mapsto tx$ be a semiflow on a compact $\textrm{T}_2$-space $X$ with phase semigroup $T$ such that $tX=X$ for each $t$ of $T$. An $x\in X$ is called an \textit{a.a. point} if $t_nx\to y, x_n^\prime\to x^\prime$ and $t_nx_n^\prime=y$ implies $x=x^\prime$ for every net $\{t_n\}$ in $T$.
In this paper, we study the a.a. dynamics of $(T,X)$; and moreover, we present a complete proof of Veech's structure theorem for a.a. flows.
\end{abstract}

\begin{keyword}
Semiflow; almost automorphy; locally almost periodic; almost C-semigroup.

\medskip
\MSC[2010] 37B05 $\cdot$ 37B20 $\cdot$ 20M20
\end{keyword}
\end{frontmatter}


\section{Introduction}\label{sec0}
\begin{sdt}
Throughout, unless specified otherwise, we assume:
\begin{enumerate}
\item $X$ is a non-empty compact $\textrm{T}_2$ space with the compatible symmetric uniform structure $\mathscr{U}_X$.

\item Given $x\in X$ and $\varepsilon\in\mathscr{U}_X$, write
$\varepsilon[x]=\{y\in X\,|\,(x,y)\in\varepsilon\}$.

\item $\varDelta_X=\{(x,x)\,|\,x\in X\}$ is the diagonal of $X\times X$.

\item $T$ is a topological semigroup with identity $e$.
\end{enumerate}
In fact, $\mathscr{U}_X$ is exactly the family of the symmetric open neighborhoods of $\varDelta_X$ in the product space $X\times X$.
\end{sdt}
\begin{description}
\item[A.] We will say that a jointly continuous map $(t,x)\mapsto tx$ of $T\times X$ to $X$ is the phase mapping of a \textit{semiflow} with phase space $X$ and with phase semigroup $T$, denoted $(T,X)$, if
$ex=x$ for all $x\in X$ and $t(sx)=(ts)x$ for all $s,t\in T, x\in X$.
Here $(T,X)$ will be called a \textit{flow} with phase group $T$ when $T$ is a group. Based on $(T,X)$, $Tx=\{tx\,|\,t\in T\}$ is called the \textit{orbit} or \textit{motion} with initial state $x\in X$.

\item[B.] As usual, $(T,X)$ will be called a \textit{minimal semiflow} if and only if $\mathrm{cls}_XTx=X$ for all $x\in X$; and we could similarly define \textit{minimal subsets} of $(T,X)$. See, e.g., \cite{G, Fur, EEN, AD}. Since $X$ is compact here, the minimality will be equivalently described by the almost periodicity later.
\end{description}

In this paper, we shall be mostly concerned with a kind of dynamical system---``invertible semiflow'' which lies between flow and semiflow.

\begin{sn}[{cf.~\cite{AD}}]\label{sn0.1}
Let $(T,X)$ be a semiflow. Then:
\begin{enumerate}
\item $(T,X)$ is said to be \textit{surjective} if $tX=X$ for all $t\in T$. For example, when $(T,X)$ is minimal admitting an invariant Borel probability measure, then it is surjective \cite[Proposition~3.6]{AD}.

\item $(T,X)$ is called \textit{invertible} if for each $t\in T$ the transition $x\mapsto tx$ is an invertible self-map of $X$.
      In this case, $(x,t)\mapsto xt:=t^{-1}x$ of $X\times T$ to $X$, denoted $(X,T)$, is a right-action semiflow \cite[Lemma~0.6]{AD}, which is called the \textit{reflection} of $(T,X)$.
    \begin{itemize}
    \item A ``dynamical'' property is said to satisfy the \textit{reflection principle} in case this property holds for $(T,X)$ if and only if it holds for $(X,T)$.
    \end{itemize}

\item When $(T,X)$ is invertible, by $\langle T\rangle$ we denote the smallest discrete group of self homeomorphisms of $X$ with $T\subseteq\langle T\rangle$ where we have identified each $t$ in $T$ with the transition $x\mapsto tx$.
\end{enumerate}

Clearly, $(\langle T\rangle, X)$ is a flow with discrete phase group $\langle T\rangle$, which has distinct dynamics with $(T,X)$ in general (cf.~\cite{EEN, AD} for examples).
\end{sn}
\begin{description}
\item[C.] A minimal flow $(T,X)$ is called ``almost automorphic'' (a.a. for short) if there exists a point $x\in X$ such that $t_nx\to y$ and $t_n^{-1}y\to x^\prime$ implies $x=x^\prime$ for all net $\{t_n\}$ in $T$. This important notion, as a generalization of almost periodic point, was first introduced by S.~Bochner in 1955. Then fundamental properties of a.a. functions on groups and a.a. flows on compact $\textrm{T}_2$ spaces have been systematically studied by Veech (cf.~\cite{V63, V65, V77}) and then others (see, e.g., \cite{MW, Fur, SY, SSY, Mil, AM, AGN} and so on).
\end{description}

In this paper, we shall consider ``almost automorphy'' of surjective/invertible semiflows. In particular, we will present a complete self-contained proof of Veech's structure theorem for a.a. flows and a.a. abelian semiflows. Moreover, we will introduce a kind of phase semigroup\,---\,``almost C-semigroup'' that is fruitful for locally almost periodic semiflows.

\subsection{Three natural invertible semiflows}\label{sec0.1}
\begin{enumerate}
\item Let $\varphi\colon\mathbb{Z}\times X\rightarrow X$ be a discrete-time flow. Then it naturally induces an invertible semiflow $\varphi_+\colon\mathbb{Z}_+\times X\rightarrow X$ by $(t,x)\mapsto\varphi(t,x)$.

\item Let $\varphi\colon \mathbb{R}\times X\rightarrow X$ be a continuous-time flow. Then it naturally induces an invertible semiflow $\varphi_+\colon\mathbb{R}_+\times X\rightarrow X$.

\item Let $G$ be a discrete group and $T$ a subsemigroup of $G$ with $e\in T$ like $(G,T)=(\mathbb{Z},\mathbb{Z}_+)$ and  $(G,T)=(\mathbb{R},\mathbb{R}_+)$;
and let $L^\infty(G)$ be the set of all bounded $\mathbb{C}^d$-valued functions on $G$ endowed with the pointwise topology, i.e., a net $f_n\to f$ if and only if $f_n(x)\to f(x)\ \forall x\in G$, where $d\ge1$.
\begin{enumerate}
\item We can define the right translate semiflow
$\tau_R\colon T\times L^\infty(G)\rightarrow L^\infty(G)$ by $(t,f)\mapsto tf$,
where for every $t\in T$, $tf(x)=f(x t)$ for all $x\in G$.
Here the phase space $L^\infty(G)$ is not compact in general.

Given any $\xi\in L^\infty(G)$, define $H_G(\xi)$, called the hull of $\xi$, to be the pointwise closure of $G\xi$. Since $\xi$ is bounded, $H_G(\xi)$ is a compact $\textrm{T}_2$ subspace of $L^\infty(G)$. Clearly, $(T,H_G(\xi))$ is an invertible semiflow with phase semigroup $T$. However, $(T,H_T(\xi))$ need not be surjective and so not invertible.

\item Similarly, we can define a left-translate invertible semiflow $(L^\infty(G), T)$ with the same phase semigroup $T$ as follows:
$\tau_L\colon L^\infty(G)\times T\rightarrow L^\infty(G)$ by $(f,t)\mapsto ft$ where $ft(x)=f(tx)$ for all $x\in G$.

\item Moreover, $(T,L^\infty(G),T)$ is a \textit{bitransformation semigroup} in the sense that $(sf)t=s(ft)$ for all $s,t\in T$ and $f\in L^\infty(G)$.
\end{enumerate}

\textbf{(i)} Motivated by \cite[Definition~1.2.1]{V65} that is for $T=G$, here $\xi\in L^\infty(G)$ is called a right (resp.~left) \textit{$T$-a.a. function} on $G$ if for all net $\{t_n\}$ in $T$, $t_n\xi\to f\in L^\infty(G)$ implies $t_n^{-1}f\to\xi$ (resp. $\xi t_n\to f\in L^\infty(G)$ implies $ft_n^{-1}\to\xi$).

\begin{sthm}[{cf.~\cite[Theorem~1.3.1]{V65} for $T=G$}]
Let $f\in L^\infty(G)$. Does it hold that $f$ is right $T$-a.a. if and only if it is left $T$-a.a. in $L^\infty(G)$?
\end{sthm}
\begin{proof}[Proof for $T=G$] (Veech~\cite{V65})
Suppose $f$ is right $T$-a.a. and let for a net $\{t_n\,|\,n\in\Lambda\}$ in $T$, $f t_n\to g$ and $gt_n^{-1}\to h$ for some $g,h\in L^\infty(G)$. We must prove $f=h$. To this end, let $x\in G$ be any fixed. We will show $f(x)=h(x)$. At first, since $\{t_n^{-1}x\}$ is a net in $T$ for $T=G$, without loss of generality we can assume $(t_n^{-1}x)f\to g_1^{}$ and $(t_n^{-1}x)^{-1}g_1^{}\to f$ for some $g_1^{}\in L^\infty(G)$. Now for each finite subset $N$ of $G$ of cardinality $|N|$ we can choose $j$ and then $k$ in $\Lambda$ such that both
\begin{equation*}
\max_{t\in N}\left|f(t)-(t_j^{-1}x)^{-1}(t_k^{-1}x)f(t)\right|<\frac{1}{|N|}\quad \left(\textrm{or } \max_{t\in N}\left|f(t)-f(tx^{-1}t_jt_k^{-1}x)\right|<\frac{1}{|N|}\right)\leqno{(0.1.1)}
\end{equation*}
and
\begin{equation*}
\left|h(x)-ft_kt_j^{-1}(x)\right|<\frac{1}{|N|}\quad \left(\textrm{or }\left|h(x)-f(xx^{-1}t_kt_j^{-1}x)\right|<\frac{1}{|N|}\right)\leqno{(0.1.2)}
\end{equation*}
Let $\delta_N=x^{-1}t_jt_k^{-1}x$, $k=k(N)$, $j=j(N)$. Then $\{\delta_N\}$ is a net indexed by the system of finite subsets of $T$ directed by inclusion. From (0.1.1) we have $\delta_Nf\to f$ so that $\delta_N^{-1}f\to f$. It follows from (0.1.2) that $\delta_N^{-1}f(x)\to h(x)$. So $f(x)=h(x)$. Since $x\in G$ is arbitrary, $f=h$.
\end{proof}

We note that even if $T$ is a proper subgroup of $G$, $\{t_n^{-1}x\}$ in the foregoing proof need not be a net in $T$. Thus the above question is not obvious for the case $T\subsetneq G$.

\medskip
\textbf{(ii)} Recall~\cite{Neu, B62, V65} that $\xi\in L^\infty(G)$ is call an right (resp. left) \textit{$T$-a.p. function} on $G$ if $T\xi$ (resp. $\xi T$) is relatively compact in $(L^\infty(G),\|\cdot\|)$ where $\|\cdot\|$ is the supremum norm.

\begin{sthm-b}
Let $f\in L^\infty(G)$. Does it hold that $f$ is right $T$-a.p. if and only if it is left $T$-a.p. in $L^\infty(G)$?
\end{sthm-b}
\begin{proof}[Proof for $T=G$] (Due to J. von Neumann~\cite{Neu}; also see \cite[Lemma~41B]{L}) This follows from the following
\begin{thm*}[\cite{Neu}]
If $f\in L^\infty(G)$ is left $T$-a.p. and $\epsilon>0$, there exists a finite set of points $\{b_k\}$ in $T$ such that given any $b\in T$ one of the points $b_k$ can be chosen so that $|f(xby)-f(xb_ky)|<\epsilon$ for every $x$ in $T$ and $y$ in $G$.
\end{thm*}
Indeed, let the finite set $\{fa_i\}$ be $\epsilon/4$-dense in $fT$. Given $b\in T$ and $i$, $fa_ib$ is within a distance $\epsilon/4$ of $fa_j$ for some $j$, and letting $i$ vary we obtain an integer-valued function $i\mapsto j(i)$ such that $\|fa_ib-fa_{j(i)}\|<\epsilon/4$ for every $i$. For each such integer-valued function $j$ let us choose one such $b\in T$ (if there is any) and let $\{b_j\}$ be the finite set so obtained. Then, by the very definition of $\{b_j\}$, for every $b\in T$ there exists one of the points $b_k$ such that $\|fa_ib-fa_ib_k\|<\epsilon/2$ for all $i$. Since for any $x\in T$ we can find $a_i$ so that $\|fx-fa_i\|<\epsilon/4$, we have
\begin{gather*}
\|fxb-fxb_k\|\le\|fxb-fa_ib\|+\|fa_ib-fa_ib_k\|+\|fa_ib_k-fxb_k\|<\frac{\epsilon}{4}+\frac{\epsilon}{2}+\frac{\epsilon}{4}=\epsilon.
\end{gather*}
That is, for every $b\in T$ there exists one of the points $b_k$ such that
\begin{gather*}
|f(xby)-f(xb_ky)|<\epsilon\quad \forall x\in T\textrm{ and }y\in G.
\end{gather*}
Therefore, if $f$ is left $G$-a.p. then it is a right $G$-a.p. function on $G$.
\end{proof}
\end{enumerate}

The above two questions indicate that some important properties of functions on groups cannot be ``obviously'' extended to general flows/semiflows. In particular Question~0.1(i) shows that Veech's classical theory of a.a. functions on groups cannot be applied ``mechanically'' to general flows/semiflows.
\subsection{Basic notation}\label{sec0.2}
In order to precisely state our main theorems we will prove in this paper, we need to introduce and recall some basic notions in preparation.

\begin{sn}[{cf.~\cite{G, Fur, EEN, CD, AD}}]\label{sn0.2}
Let $(T,X)$ be a semiflow, $A\subset T$, and $x\in X$. Then:
\begin{enumerate}
\item $A$ is called (right-)\textit{syndetic} in $T$ if there is a compact subset $K$ of $T$ with $Kt\cap A\not=\emptyset$ for every $t\in T$.
(Notice that a syndetic set need not be ``relatively dense'' in $T$ in the sense of Veech (cf.~\cite[Definition~2.1.1]{V65} and also see Definition~\ref{sn2.11}).

\item The $x$ is called \textit{almost periodic} (a.p.) for $(T,X)$ if for every neighborhood $U$ of $x$ in $X$,
\begin{equation*}
N_T(x,U)=\{t\in T\,|\,tx\in U\}
\end{equation*}
is syndetic in $T$. If every point of $X$ is a.p., then we say $(T,X)$ is a \textit{pointwise a.p.} semiflow.

\item $(T,X)$ is called \textit{almost periodic} (a.p.) if given $\varepsilon\in\mathscr{U}_X$ there is a syndetic set $A$ in $T$ such that $Ax\subseteq\varepsilon[x]$ for all $x\in X$.
\end{enumerate}

Since {\it $x$ is an a.p. point of $(T,X)$ if and only if $\mathrm{cls}_XTx$ is a minimal set of $(T,X)$} \cite{G,Fur,AD}, whether or not $x$ is an a.p. point does not depend upon the topology on $T$. Moreover, any a.p. semiflow is a pointwise a.p. semiflow.
\end{sn}

\begin{sn}[{cf.~\cite{H44,GH}}]\label{sn0.3}
Let $(T,X)$ be a semiflow. Then:
\begin{enumerate}
\item An $x\in X$ is said to be \textit{locally almost periodic} (l.a.p. for short) if for all neighborhood $U$ of $x$ there exist a neighborhood $V$ of $x$ and a syndetic subset $A$ of $T$ with $AV\subseteq U$.

\item If $(T,X)$ is pointwise l.a.p., then $(T,X)$ is referred to as an \textit{l.a.p. semiflow}.
\end{enumerate}

Clearly, an l.a.p. point is an a.p. point. Moreover, if $(T,X)$ is a.p., then it is l.a.p. invertible
(cf.~\cite{DX, AD}).
\end{sn}

Recall that $T$ is called a ``right C-semigroup'' if and only if $\textrm{cls}_T(T\setminus Tt)$ is compact in $T$ for all $t$ in $T$ (cf.~\cite{KM,AD}).
Now we will introduce a kind of more general topological semigroup than C-semigroup.

\begin{sn}\label{sn0.4}
$\,$
\begin{enumerate}
\item $T$ is called an \textit{almost right C-semigroup} if and only if
$\{t\in T\,|\,\textrm{cls}_T(T\setminus Tt)\textrm{ is compact in }T\}$
is dense in $T$.
\item We could similarly define \textit{almost left C-semigroup}.
\item If $T$ is not only an almost right C-semigroup but also an almost left C-semigroup, then it is referred to as an \textit{almost C-semigroup}.
\end{enumerate}
\end{sn}

Clearly every topological group is a right C-semigroup and any right C-semigroup is an almost right C-semigroup like $(\mathbb{R}_+,+)$ and $(\mathbb{Z}_+,+)$ with the usual topologies. In fact, there are the following important almost right C-semigroups which are not right C-semigroups.

\begin{exas}\label{0.5}
We now construct some almost C-semigroups which are not C-semigroups.
\begin{enumerate}
\item Let $\mathbb{R}^{n\times n}$ be the space of all real $n\times n$ matrices endowed with the usual topology. Since the nonsingular matrices are open dense in $\mathbb{R}^{n\times n}$, thus $T=(\mathbb{R}^{n\times n},\circ)$ is an almost C-semigroup.

\item Let $M^n$ be a compact boundaryless Riemannian manifold of dimension $n$ and let $T$ be the semigroup of $C^0$-endomorphisms of $M^n$ with the $C^0$-topology. Since $\textrm{Diff}^1(M^n)$ is dense in $T$, thus $T$ is an almost C-semigroup under the composition of maps.

\item It is known that $\textrm{Diff}^1(M^n)$ is open in $C^1(M^n,M^n)$ under the $C^1$-topology. Let $T$ be the $C^1$-closure of $\textrm{Diff}^1(M^n)$ in $C^1(M^n,M^n)$. Then by the chain rule, $T$ is an almost C-semigroup under the $C^1$-topology.
\end{enumerate}
\end{exas}

When $T$ is discrete, then it is a right C-semigroup iff it is an almost right C-semigroup. Next it is easy to verify the following basic fact (cf., e.g., \cite[Theorem~4.11]{GH} and \cite[p.~98]{AM} for $T$ a group).

\begin{lem}\label{lem0.6}
Let $(T,X)$ be any semiflow with $T$ an almost right C-semigroup or an abelian semigroup. If $(T,X)$ is l.a.p. at some point $x_0\in X$, then each of $\mathrm{cls}_XTx_0$ is an l.a.p. point of $(T,X)$.
\end{lem}

\begin{proof}
First we write $G=\{t\in T\,|\,\textrm{cls}_T(T\setminus Tt)\textrm{ is compact in }T\}$, which is dense in $T$ if $T$ is an almost right C-semigroup.
Let $x\in\textrm{cls}_X{Tx_0}$ and $U$ a neighborhood of $x$. Then there is an $s\in T$ such that $sx_0\in U$ so that $x_0\in s^{-1}U$. Further there exist a syndetic set $A\subset T$ and an open neighborhood $V^\prime$ of $x_0$ such that $AV^\prime\subset s^{-1}U$ and so $sAV^\prime\subset U$. Because $\textrm{cls}_X{Tx_0}$ is a minimal set of $(T,X)$, there is a $\tau\in T$ (with $\tau\in G$ if $T$ is an almost right C-semigroup) such that $\tau x\in V^\prime$, so there is an open neighborhood $V$ of $x$ such that $\tau V\subseteq V^\prime$. Then $(sA\tau)V\subseteq U$ and $sA$ is syndetic in $T$.

If $T$ is an abelian semigroup, then $sA\tau=s\tau A$ is syndetic in $T$, so $x$ is an l.a.p. point of $(T,X)$ by Definition~\ref{sn0.3}.
Next, suppose $T$ is an almost right C-semigroup and set $C=\textrm{cls}_T(T\setminus T\tau)$, which is compact in $T$ by Definition~\ref{sn0.4}. Since $x$ is also a discretely  a.p. point of $(T,X)$,
there is a finite subset $K$ of $T$ such that for each $t\in T$, $Kt\cap N_T(x,U)\not=\emptyset$. Because $T\times X\rightarrow X$ is jointly continuous and $C$ is compact, we can find a $\delta\in\mathscr{U}_X$ with $\delta[x]\subseteq V$ such that for each $t\in T\setminus T\tau$ there is some $k=k(t)\in K$ with $kt(\delta[x])\subset U$. Then $\{t\in T\,|\,t(\delta[x])\subseteq U\}$ is syndetic in $T$ and thus $x$ is an l.a.p. point of $(T,X)$.
The proof is complete.
\end{proof}

\begin{sn}\label{sn0.7}
Let $(T,X)$ be any semiflow and $x\in X$. Then:
\begin{enumerate}
\item The $x$ is called an \textit{equicontinuous point} of $(T,X)$, denoted $x\in \textrm{Equi}\,(T,X)$, if given $\varepsilon\in\mathscr{U}_X$, there is a $\delta\in\mathscr{U}_X$ such that $(x,x^\prime)\in\delta$ implies $(tx,tx^\prime)\in\varepsilon$ for all $t\in T$.

\item If $\textrm{Equi}\,(T,X)=X$, then $(T,X)$ is called \textit{equicontinuous}.
\end{enumerate}

By Definitions~\ref{sn0.2}, \ref{sn0.3} and \ref{sn0.7} we can see that
``\textit{if $x\in\textrm{Equi}\,(T,X)$ and $x$ is a.p., then $x$ is an l.a.p. point of $(T,X)$}.''
Since $X$ is compact here, then ``\textit{$(T,X)$ is equicontinuous iff given $\varepsilon\in\mathscr{U}_X$ there exists $\delta\in\mathscr{U}_X$ such that $(tx,ty)\in\varepsilon$ for all $t\in T$ whenever $x,y\in X$ with $(x,y)\in\delta$}'' (cf.~\cite[Lemma~1.1]{AD}).
Moreover, any semiflow $(T,X)$ is a.p. iff it is equicontinuous invertible (cf.~\cite{DX,AD}).
Thus the a.p. property of a semiflow is also independent of the topology of the phase semigroup $T$.
\end{sn}

\begin{lem}[{cf.~\cite[p.~35]{G03} for $T$ a group and \cite[7.5(1)]{AD} for $T$ a right C-semigroup}]\label{lem0.8}
Suppose $(T,X)$ is with $T$ an almost right C-semigroup. If $x_0\in\mathrm{Equi}\,(T,X)$ is an a.p. point, then $\mathrm{cls}_XTx_0\subseteq\mathrm{Equi}\,(T,X)$.
\end{lem}

\begin{proof}
Let $x\in\textrm{cls}_XTx_0$ and $\varepsilon\in\mathscr{U}_X$. There is an $\alpha\in\mathscr{U}_X$ such that if $y\in\alpha[x_0]$ then $t(x_0,y)\in\varepsilon$ for all $t\in T$. Since $x_0$ is a.p., there is some $s\in T$ with $sx\in\alpha[x_0]$ so that there is an $s$ in $\{t\in T\,|\,\textrm{cls}_T(T\setminus Tt)\textrm{ is compact in }T\}$ with this property and further there exists some $\beta\in\mathscr{U}_X$ such that $s(\beta[x])\subseteq\alpha[x_0]$. Thus, if $y\in\beta[x]$ then $ts(x,y)\in\varepsilon$ for all $t\in T$. Moreover, since $\textrm{cls}_T(T\setminus Ts)$ is compact in $T$, there is some $\delta\in\mathscr{U}_X$ with $\delta\subset\beta$ such that if $(x,y)\in\delta$ then $t(x,y)\in\varepsilon$ for all $t\in T\setminus Ts$ and thus $t(x,y)\in\varepsilon$ for all $t\in T$. This shows $x\in\textrm{Equi}\,(T,X)$.
\end{proof}

However, it should be noted that even if $(T,X)$ is \textit{invertible}, we do not know if $(T,\textrm{cls}_XTx_0)$ is a.p. or not in the situation of Lemma~\ref{lem0.8} where $T$ is not a group; see \cite[(3) of Remark~3.13]{AD} for a counterexample in the case where $(T,X)$ is not invertible.

\begin{sn}[{cf.~Veech~\cite[p.~741]{V65} for $T$ a group}]\label{sn0.9}
Let $(T,X)$ be a semiflow. Then:
\begin{enumerate}
\item An $x$ in $X$ is said to be \textit{almost automorphic} (a.a. for short), denoted $x\in P_{\!aa}(T,X)$, in case for all net $\{t_n\}$ in $T$,
$t_nx\to y$, $x_n^\prime\to x^\prime$ and $t_nx_n^\prime=y$ implies $x=x^\prime$.

\item If $x\in P_{\!aa}(T,X)$ and $\mathrm{cls}_XTx=X$, then $(T,X)$ will be called an \textit{a.a. semiflow}.

\item If $P_{\!aa}(T,X)=X$, then $(T,X)$ is called a \textit{pointwise a.a.} semiflow.
\end{enumerate}

We will show that every a.a. semiflow is point-distal (cf.~Lemma~\ref{1.8} in $\S\ref{sec1}$).
Now we can conclude the following basic properties.
\begin{itemize}
\item Let $(T,X)$ be a surjective semiflow. Then:
\begin{enumerate}
\item[(i)] The set $P_{\!aa}(T,X)$ is invariant; that is, if $x\in P_{\!aa}(T,X)$ then $tx\in P_{\!aa}(T,X)$ for all $t\in T$.

\item[(ii)] If $x\in P_{\!aa}(T,X)$ and $t_nx\to y$, then $t_n^{-1}[y]\to \{x\}$ in the sense of Hausdorff topology.

\item[(iii)] If $x\in P_{\!aa}(T,X)$ where $(T,X)$ is not necessarily surjective, then $(T,X)$ is 1-1 at $x$, i.e., $t^{-1}[tx]=\{x\}$ for all $t\in T$.
\end{enumerate}

\begin{proof}
$\,$\\
(i). Let $x\in P_{\!aa}(T,X)$ and $t\in T$. Assume $t_n(tx)\to y$, $x_n^{\prime\prime}\to x^{\prime\prime}$ and $t_nx_n^{\prime\prime}=y$. We need to show $x^{\prime\prime}=tx$. For this, take $x_n^\prime\in t^{-1}[x_n^{\prime\prime}]$ and we can assume $x_n^\prime\to x^\prime$ so $tx^\prime=x^{\prime\prime}$.
Since $(t_nt)x\to y, x_n^\prime\to x^\prime$ and $(t_nt)x_n^\prime=y$, we have $x=x^\prime$ and then $tx=tx^\prime=x^{\prime\prime}$. This shows $tx\in P_{\!aa}(T,X)$.

(ii). Let $x\in P_{\!aa}(T,X)$ and $t_nx\to y$. If $x_n^{\prime\prime}\in t_n^{-1}[y]$, then $x_n^{\prime\prime}\to x$. This implies that $t_n^{-1}[y]\to \{x\}$.

(iii). Let $t\in T$ and $x\in P_{\!aa}(T,X)$. For $t_n=t$, $t_nx\to y=tx$ implies $t^{-1}[tx]=t_n^{-1}[y]\to \{x\}$. Thus $t^{-1}[tx]=\{x\}$. The proof is complete.
\end{proof}
\end{itemize}

Clearly, an a.a. semiflow need not be a pointwise a.a. semiflow. It will be proven later that
\begin{itemize}
\item {\it A minimal surjective semiflow $(T,X)$ is pointwise a.a. if and only if $(T,X)$ is equicontinuous invertible} (cf.~Theorem~\ref{thm4.5}).
\end{itemize}
This theorem in the flows is due to Veech~\cite{V65,V77} by using his structure theorem (Theorem~\ref{0.13} below) and also see \cite[Theorem~7]{AGN} by using Ellis' semigroup.

When $(T,X)$ is a flow and $x\in\textrm{Equi}\,(T,X)$ is a.p., then $\textrm{cls}_XTx\subseteq P_{\!aa}(T,X)$. However, this is not obvious in the invertible semiflows (cf.~Lemma~\ref{lem0.8}).
\end{sn}

Let $(T,X)$ and $(T,Y)$ be two semiflows with the same phase semigroup $T$. $\pi\colon(T,X)\rightarrow(T,Y)$ is called an ``epimorphism'' if $\pi\colon X\rightarrow Y$ is a continuous surjective map such that $\pi(tx)=t\pi(x)$ for all $x\in X$ and $t\in T$. In this case, $(T,Y)$ is called a factor of $(T,X)$ and $(T,X)$ an extension of $(T,Y)$.

\begin{sn}
An epimorphism $\pi\colon(T,X)\rightarrow(T,Y)$ of two minimal semiflows is said to be of \textit{almost 1-1 type} if there exists $y\in Y$ such that $\pi^{-1}[y]$ is a singleton set. In this case, $(T,X)$ is called an \textit{almost 1-1 extension} of $(T,Y)$.
\end{sn}

\begin{sn}[{cf.~\cite{AD}}]\label{0.11}
Let $(T,X)$ be any semiflow and $x,y\in X$. Then:
\begin{enumerate}
\item We say $x$ is \textit{proximal} to $y$, denoted $(x,y)\in P(X)$ or $y\in P[x]$, if there is a net $\{t_n\}$ in $T$ with $\lim_nt_nx=\lim_nt_ny$.

\item We say $x$ is \textit{regionally proximal to} $x^\prime$, denoted $(x,x^\prime)\in Q(X)$ and $x^\prime\in Q[x]$, if there are nets $\{x_n\}, \{x_n^\prime\}$ in $X$ and $\{t_n\}$ in $T$ such that $x_n\to x, x_n^\prime\to x^\prime$ and $\lim_n t_nx_n=\lim_nt_nx_n^\prime$.

\item We say that $(x,y)\in Q^{-}(T,X)$ provided that given $\alpha\in\mathscr{U}_X$ and neighborhoods $U$ of $x$ and $V$ of $y$ there exist $x_1\in U$, $y_1\in V$ and $t\in T$ with $(t^{-1}[x_1]\times t^{-1}[y_1])\cap\alpha\not=\emptyset$. Then
    \begin{gather*}
    Q^{-}(T,X)={\bigcap}_{\alpha\in\mathscr{U}_X}\mathrm{cls}_{X\times X}T\alpha
    \end{gather*}
\end{enumerate}

Since $X$ is a $\textrm{T}_2$-space, it is easy to check that $x,y\in P_{\!aa}(T,X)$ and $x\not=y$ implies $(x,y)\not\in P(X)$. We note here that although $Q(T,X)$ need not be $T$-invariant for $T$ is only a semigroup, yet $Q^{-}(T,X)$ is a closed invariant subset of $(T,X\times X)$.
\end{sn}

Clearly, if $Q[x]=\{x\}$ then $x\in P_{\!aa}(T,X)$ by Definition~\ref{sn0.9}. By
$Q(X)={\bigcap}_{\alpha\in\mathscr{U}_X}\mathrm{cls}_{X\times X}{\bigcup}_{t\in T}t^{-1}\alpha$, $Q(X)$ is a closed reflexive symmetric relation on $X$, but $P(X)$ and $Q(X)$ both are neither transitive nor invariant in general semiflows. However, we can obtain the following.

\begin{thm}[{cf.~\cite[Lemma~1.10]{AD} for $T$ a right C-semigroup}]
If $(T,X)$ is an invertible semiflow with $T$ an almost right C-semigroup, then $Q(T,X)$ is invariant.
\end{thm}

\begin{proof}
At first, if $t\in T$ is such that $T\setminus Tt$ is relatively compact in $T$, then $tQ(T,X)\subseteq Q(T,X)$. Since $Q(T,X)$ is closed and $(t,x)\mapsto tx$ is continuous, hence $TQ(T,X)\subseteq Q(T,X)$. The proof is complete.
\end{proof}

\subsection{Main statements}\label{sec0.3}
\subsubsection{Veech's structure theorem}
The first main purpose of this paper is to prove the following Theorem~\ref{0.13}, which is Veech \cite[Theorem~2.1.2]{V77} (also cf.~\cite[p.~741]{V65}). Although this theorem is of fundamental importance in the theory of a.a. dynamics, yet no complete proof has been available in the literature since Veech first claimed it in 1965~\cite{V65}.

\begin{thm}[Veech's structure theorem for a.a. flows]\label{0.13}
If $(T,X)$ is an a.a. flow, then $(T,X)$ is an almost 1-1 extension of an equicontinuous flow $(T,Y)$ via $\pi\colon(T,X)\rightarrow(T,Y)$ such that $P_{\!aa}(T,X)=\{x\in X\,|\,\pi^{-1}[\pi(x)]=\{x\}\}$.
\end{thm}

\begin{note}
It turns out that Veech's structure theorem implies that an a.a. flow is l.a.p.~\cite{V65, MW}; and \textit{if a minimal flow is pointwise a.a. then it is equicontinuous}. The latter and Theorem~\ref{0.13} together is called the ``Principal Structure Theorem'' of a.a. flows (cf.~\cite[Theorem~2.1.2]{V77}).
\end{note}

In Veech~1965~\cite{V65}, he developed a series of methods for studying a.a. functions on a discrete group $G$ and he proved Theorem~\ref{0.13} in the special case that $T=G$ and $X=H_G(\xi)$ for an a.a. function $\xi$ on $G$ as in $\S\ref{sec0.1}$. However, Veech's argument needs the basic property that a function on $G$ is right a.a. iff it is left a.a. (cf.~$\S$\ref{sec0.1}(i)). Since $(G,L^\infty(G),G)$ as in $\S\ref{sec0.1}$ with $T=G$ is a bitransformation group, then if $f\in H_G(\xi)$ is a.a. $N_G(f,U)$ is ``bi-syndetic'' in the much more stronger sense that $G=KN_G(f,U)=N_G(f,U)K$ for some finite subset $K$ of $G$ (cf.~\cite[Theorem~2.2.1]{V65}).

However, in our general flow setting, although $(T,X)$ gives rise to a natural right action flow $(X,T)$ by $(x,t)\mapsto xt$ as in Definition~\ref{sn0.1}, yet $(T,X,T)$ is not a bitransformation group when $T$ is not abelian because $t^{-1}(sx)\not=s(t^{-1}x)$ in general for $s\not=t$ in $T$. Moreover, a right-syndetic set in the sense of Definition~\ref{sn0.2} is in general not left-syndetic. Thus, from a.a. function on a group to a.a. flow on a compact $\textrm{T}_2$ space, there are some essential differences.

In view of these reasons, our proof of Theorem~\ref{0.13} is actually non-trivial. In \cite[Theorem~19]{AM}, Auslander and Markley proved Theorem~\ref{0.13} in the special case that $T$ is an abelian group and $X$ a compact metric space. They pointed out that ``it is not easy to extend \cite[Theorem~19]{AM} to flows on compact Hausdorff spaces'' even for $T$ is an abelian group (cf.~\cite[p.~107]{AM}).

We will prove Theorem~\ref{0.13} in $\S\ref{sec3}$ by firstly proving (without using Theorem~\ref{0.13}) the important fact that:
``\textit{An a.a. flow is an l.a.p. flow}'' (cf.~Theorem~\ref{thm2.5} in $\S\ref{sec2}$).
To prove this, as in \cite{V65, AM} we will show the ``two syndetic sets'' condition $ABx\subset U_x$ where $A,B$ right-syndetic sets in $T$ and $U_x$ a neighborhood of $x$ (cf.~Corollary~\ref{1.11}). In fact, we can first show the ``weak'' two syndetic sets condition $A^{-1}Bx\subset U_x$ for a kind of semiflows so that $A^{-1}$ is left-syndetic in $T$; however, if $T$ is abelian, then $A^{-1}$ is also right-syndetic and thus $x$ is an l.a.p. point. In addition, comparing with \cite{V65, AM}, no metric on $X$ here will cause some trouble. To overcome this, we will exploit the continuous pseudo-metrics $d$ on $X$, although $(X,d)$ is not a $\textrm{T}_2$-space (cf.~Proof of Theorem~\ref{1.10}).

Moreover, the converse of Theorem~\ref{0.13} can be generalized to minimal semiflows with any phase semigroups; see Theorem~\ref{thm3.6} in $\S\ref{sec3}$.
\subsubsection{Almost periodic and almost automorphic functions}
In $\S\ref{sec4}$ we will simply prove a classical theorem of Bochner which equivalently describe the almost periodic functions by the almost automorphy on a discrete group. See Theorem~\ref{thm4.10} below.

\subsubsection{Veech's relations}
The another main purpose of this paper is to study Veech's relations $V$ and $D$ on $X$ (cf.~Definitions~\ref{def5.1} and \ref{def5.4} in $\S\ref{sec5}$).
Veech in \cite{V68} proved that when $T$ is an ``abelian'' group, $\textrm{cls}_X{V[x]}=D[x]$ and moreover if $X$ is a compact metric space $V[x]=D[x]$. We will
show  in $\S\ref{sec5}$:
\begin{quote}
\textbf{Theorem}. {\it Let $(T,X)$ be an invertible minimal semiflow with any phase semigroup $T$ not necessarily abelian and let $x\in X$. Then $V[x]$ is dense in $D[x]$; and moreover, $V[x]=D[x]$ if $X$ is a compact metric space or if $T$ is countable.}
\end{quote}

Moreover, by using Veech's relations, we can further characterize almost automorphy by the regional proximity on $X$ as follows:

\begin{thm}[{cf.~\cite{V68} for $T$ an abelian group}]\label{0.14}
Let $(T,X)$ be a minimal flow. Then $(T,X)$ is an a.a. flow if and only if there exists a point $x\in X$ such that $Q[x]=\{x\}$.
\end{thm}

It should be noted that in view of Ellis' ``two-circle'' minimal set the condition that $Q[x]$ is finite non-single for all $x\in X$ does not need to imply that $(T,X)$ is an a.a. flow.

When $T$ is an ``abelian'' group, Theorem~\ref{0.14} is actually \cite[$(\textrm{i})\Leftrightarrow(\textrm{ii})$ of Proposition~1.5]{Mil} and also \cite[Corollary~20]{AM}. We will prove in $\S\ref{sec5}$ this theorem after Theorem~\ref{thm5.6} that is for invertible minimal semiflows.
In fact, Theorem~\ref{0.14} still holds if ``$(T,X)$ is a minimal invertible semiflow with $T$ an abelian semigroup'' instead of ``$(T,X)$ be a minimal flow''; see Theorem~\ref{thm5.8} in $\S\ref{sec5}$.

\begin{cor}[Inheritance~I]\label{0.15}
Let $(T,X)$ be a minimal flow, $S$ a syndetic subgroup of $T$, and $x\in X$. Then $x\in P_{\!aa}(T,X)$ if and only if $x\in P_{\!aa}(S,X)$.
\end{cor}

\begin{proof}
This result follows at once from Theorem~\ref{0.14} and $Q(T,X)=Q(S,X)$ (cf., e.g.,~\cite[Lemma~4.16]{E69} and \cite[Proposition~6.1]{AD}).
\end{proof}

\begin{cor}[Inheritance~II]\label{0.16}
Let $(T,X)$ be a minimal invertible semiflow with $T$ abelian, $S$ a syndetic subsemigroup of $T$, and $x\in X$. Then $x\in P_{\!aa}(T,X)$ if and only if $x\in P_{\!aa}(S,X)$.
\end{cor}

\begin{proof}
Using Theorem~\ref{thm5.8} and \cite[Proposition~6.1]{AD} instead of Theorem~\ref{0.14} and \cite[Lemma~4.16]{E69} respectively, we can easily obtain the conclusion of Corollary~\ref{0.16}.
\end{proof}

Let $(T,X)$ be a flow. Recall that $p\in X$ is called a \textit{periodic point} of $(T,X)$ if $\{t\,|\,tp=p\}$ is a syndetic subgroup of $T$. Then by Corollary~\ref{0.15}, any periodic point of a flow is an a.a. point.

We can define periodic points for semiflows in an obvious manner. Corollary~\ref{0.16} follows that any periodic point is an a.a. point for an abelian semiflow. In fact, we can easily show any discretely periodic point is an a.a. point for any semiflow by Definition~\ref{sn0.9}.

Recall that $(X,T)$ is the reflection of an invertible $(T,X)$. Finally we will present in $\S\ref{sec5}$ a sufficient condition for almost automorphy of a point by using distality and local almost periodicity as follows.

\begin{thm}\label{0.17}
Let $(T,X)$ be an invertible minimal semiflow with $T$ an almost right C-semigroup and $x_0\in X$. Then the following two statements hold:
\begin{enumerate}
\item[$(a)$] $(X,T)$ is minimal.
\item[$(b)$] If $x_0$ is a distal l.a.p. point of $(T,X)$, then $x_0$ is an a.a. point of $(X,T)$.
\end{enumerate}
\end{thm}

Here ``distal'' at $x_0$ means $P[x_0]=\{x_0\}$ (cf.~Definition~\ref{1.7}). If $T$ is a right C-semigroup instead of an almost right C-semigroup, then $(X,T)$ is minimal by \cite[Theorem~6.31]{AD}. Here we will need a lemma (cf.~Lemma~\ref{lem5.7} in $\S\ref{sec5}$) to generalize the right C-semigroup case.

Theorem~\ref{0.17}$(a)$ is in fact a generalization of \cite[Reflection principle III]{AD}. We now conclude this introductory section with two open questions:

\begin{que}[Reflection principle]\label{0.18}
Let $(T,X)$ be an invertible minimal semiflow with $T$ an almost C-semigroup and $x_0\in X$. If $x_0$ is a distal l.a.p. point of $(T,X)$, then:
\begin{enumerate}
\item[(1)] is $x_0$ an l.a.p. point of $(X,T)$? (Note that $x_0$ is an a.a. point of $(X,T)$ by Theorem~\ref{0.17}.)
\item[(2)] is $x_0$ an.a.a. point of $(T,X)$?
\end{enumerate}
The answer of (2) of Question~\ref{0.18} is ``yes'' if so is the answer of (1). See Theorems~\ref{thm3.8} and \ref{thm5.8} for weak solutions to (2). Moreover, the minimality of $(T,X)$ is crucial here; otherwise, one can easily construct counterexamples.
\end{que}

\begin{que}[{Existence of invariant measure; \cite[Question~2.8.1]{V77}}]\label{0.19}
Let $(T,X)$ be an a.a. flow. Does $(T,X)$ have an invariant Borel probability measure? If $(T,X)$ is distal, the answer is ``yes'' and was obtained by Furstenberg by using his structure theorem in~\cite{F63}.
\end{que}
\section{Almost automorphy of surjective semiflows}\label{sec1}

\begin{sht}
We will keep the following conventions in this section unless specified otherwise.
\begin{enumerate}
\item Let $(T,X)$ with the phase mapping $(t,x)\mapsto tx$ be a \textbf{surjective} semiflow.

\item Let $T^{-1}=\{t^{-1}\,|\,t\in T\}$ and $T^{-1}\circ T=\{t^{-1}s\,|\,s,t\in T\}$, where for each $t\in T$, we identify $t$ with the transition map $t\colon x\mapsto tx$ and $t^{-1}\colon X\rightsquigarrow X$ is the inverse of $t\colon X\rightarrow X$.

\item Write $\langle T^{-1}\circ T\rangle=\{s_1^{-1}t_1\dotsm s_n^{-1}t_n\,|\,n\ge1\textrm{ and }s_i,t_i\in T\textrm{ for }i=1,\dotsc,n\}$.

\begin{itemize}
\item Here associated to $(T,X)$, $t^{-1}\colon X\rightsquigarrow X$ by $x\mapsto t^{-1}[x]$, for $t\in T$, is a upper semi-continuous set-valued map~\cite[Lemma~0.6]{AD}. That is to say, for $x\in X$ and any neighborhood $U$ of $t^{-1}[x]$, there is a neighborhood $V$ of $x$ such that $t^{-1}[y]\subset U$ for all $y\in V$.
    \begin{proof}
    Let $x_n\to x$ and we will show $t^{-1}[x_n]\subset U$ as $n$ sufficiently big. Suppose the contrary that there is a (subnet of) net
$y_n\in t^{-1}[x_n]\setminus U$ with $y_n\to y$ in $X$. Then $y\not\in t^{-1}[x]$ but
$ty_n=x_n\to x$ and $ty_n\to ty$. So $ty=x$ implies $y=t^{-1}[x]$ a contradiction.
    \end{proof}
\item In addition, $(T,X)$ need not be invertible in this section. Of course, if $(T,X)$ is an invertible semiflow here, then $\langle T^{-1}\circ T\rangle=\langle T\rangle$.
\end{itemize}
\item Let $f\colon X\rightarrow X$ be a continuous surjective map. We say the set-valued map $f^{-1}\colon X\rightsquigarrow X,\ x\mapsto f^{-1}[x]$ is \textit{continuous} if and only if $x_n\to x$ implies $f^{-1}[x_n]\to f^{-1}[x]$ in the sense of Hausdorff topology.
\end{enumerate}
\end{sht}

Let $G$ be a discrete group; then there are two canonical translate flows on $L^\infty(G)$ with the same phase group $G$ under the pointwise topology:
\begin{gather*}
\tau_R\colon G\times L^\infty(G)\rightarrow L^\infty(G),\quad(t,f)\mapsto tf
\qquad\textrm{and}\qquad
\tau_L\colon L^\infty(G)\times G\rightarrow L^\infty(G),\quad(f,t)\mapsto ft.
\end{gather*}
Clearly, $(tf)s=t(fs)$ for all $f\in L^\infty(G)$ and $s,t\in G$; that is to say, $(t,f,s)\mapsto tfs$ is a bitransformation group. Following \cite[Definition~1.2.1]{V65}, $f\in L^\infty(G)$ is called a right Bochner a.a. function on $G$ in case
$t_nf\to g$ and $t_n^{-1}g\to f^\prime$ implies $f=f^\prime$.
Similarly we could define the left Bochner a.a. function on $G$.

Then, $f\in L^\infty(G)$ is right a.a. if and only if it is left a.a. (cf.~\cite[Theorem~1.3.1]{V65} and $\S$\ref{sec0.1}(i)). This symmetric property plays an important role in Veech's arguments on a.a. functions on $G$.

However, given a flow $(T,X)$, since $ts^{-1}\not=s^{-1}t\ \forall s,t\in T$ in general, we are unable to obtain a natural bitransformation flow on $X$. This will cause many essential difficulties for our later discussion.
The principal result of the discussion in $\S\ref{sec1}$ is Theorem~\ref{1.10}, which asserts that an a.a. point of an invertible semiflow satisfies the so-called ``two syndetic sets'' condition.

Recall that $x\in P_{\!aa}(T,X)$ iff for every net $\{t_n\}$ in $T$, $t_nx\to y, x_n^\prime\to x^\prime$ and $t_nx_n^\prime=y$ implies $x=x^\prime$. Such an $x$ is also called a ``Bochner a.a. point'' of $(T,X)$ here to distinguish with the ``Bohr almost automorphy'' of $(T,X)$ that we will define in $\S\ref{sec2}$. At first we here have the following simple observation.

\begin{lem}\label{1.1}
If $(T,X)$ is a.a such that $x\mapsto t^{-1}[x]$ is continuous for $t\in T$, then $(T,X)$ is invertible.
\end{lem}

\begin{proof}
Let $x_0\in P_{\!aa}(T,X)$ such that $X=\textrm{cls}_XTx_0$. Since $t^{-1}t\colon Tx_0\rightarrow Tx_0$ is the identity and $x\mapsto t^{-1}tx$ is continuous in the Hausdorff topology, hence $t^{-1}t$ is the identity on $X$. Since $tX=X$, thus $t^{-1}$ is 1-1 onto. This proves Lemma~\ref{1.1}.
\end{proof}

\begin{sn}[{cf.~\cite{Fur} for $T=\mathbb{Z}_+$}]\label{1.2}
$\,$
\begin{enumerate}
\item A subset $D$ of $T^{-1}\circ T$ is called a \textit{$\Delta$-set} of $T$ if one can find a net $\{t_n\}$ in $T$ such that for all $m$, $t_m^{-1}t_n\in D$ as $n>n_m$ for some $n_m>m$.

\item A subset of $T^{-1}\circ T$ is called a \textit{$\Delta^*$-set} of $T$ if it intersects non-voidly every $\Delta$-set of $T$.
\end{enumerate}
 Here $D$ need not be a subset of $T$ if $T$ is not a group.
\begin{note}
A special case of Definition~\ref{1.2}.1 is \cite[(d) of Definition~1.2]{HS} where a ``$\Delta$-set'' of $T$ is a set $D\subseteq T$ such that there exists a sequence $\{s_i\}_{i=1}^\infty$ in $T$ with $s_n\in s_m D$ for all $m<n$.
\end{note}

Recall that a subset $A$ of $T$ is \textit{thick} if for every compact set $F\subset T$ there is some $t\in T$ with $Ft\subset A$. A set $H\subseteq T$ is referred to as an \textit{IP-set}~\cite{DL} if there exists a sequence $\{p_n\}_{n=1}^\infty$ in $T$ such that
$p_{n_1}p_{n_2}\dotsm p_{n_k}\in H$ for all $1\le n_1<\dotsm<n_k<\infty$ and $1\le k<\infty$.
\begin{itemize}
\item[(\ref{1.2}a)] {\it If $S$ is a thick subset of $T$, then it contains an IP-set of $T$} (cf. \cite[Lemma~9.1]{Fur} for $T=(\mathbb{Z}_+,+)$).

\begin{proof}
Let $S$ be thick and take $p_1\in S$; then there is some $p_2\in T$ with $\{e,p_1\}p_2\subseteq S$. Further there is some $p_3$ such that $\{e,p_1,p_2,p_1p_2\}p_3\subseteq S$ and so on, we can choose a sequence $\{p_n\}$ such that $p_{n_1}\dotsm p_{n_k}\in S$ for all $1\le n_1<\dotsm<n_k<\infty$. This shows that $S$ contains an IP-set of $T$. (This proof is valid for all topological monoid.)
\end{proof}
\item[(\ref{1.2}b)] {\it Let $X_0$ be an $T$-invariant subset of $X$ such that $t^{-1}t_{|X_0}=\textit{id}_{X_0}$ for all $t\in T$. If $H\subseteq T$ is an IP-set, then $H$ is a $\Delta$-set in $T$ w.r.t. $(T,X_0)$}.
\begin{proof}
Let $\{p_n\}_{n=1}^\infty$ be a sequence in $T$ with $p_{n_1}p_{n_2}\dotsm p_{n_k}\in H$ for $1\le n_1<\dotsm<n_k<\infty$ and $1\le k<\infty$. Now set $t_n=p_1p_2\dotsm p_n$ for $n=1,2,\dotsc$ and then $\{t_n\}$ is a sequence in $T$ such that
$t_n=t_mp_{m+1}\dotsm p_n$ so that ${t_m^{-1}t_n}_{|X_0}={t_m^{-1}t_mp_{m+1}\dotsm p_n}_{|X_0}={p_{m+1}\dotsm p_n}_{|X_0}\in H$ for $m<n$.
Thus $H$ is a $\Delta$-set in $T$.
\end{proof}
We notice here that if $(T,X)$ is not invertible, then ``$t_m^{-1}t_n\in H$'' generally makes no sense in the above proof. So we have to consider ${t_m^{-1}t_n}_{|X_0}$ here.
\end{itemize}
\end{sn}

It should be mentioned that a $\Delta$-set and $\Delta^*$-set of $T$ need not be a subset of $T$ if $T$ is not a group in our setting. However, any IP-set of $T$ is always a subset of $T$.

\begin{sn}[{cf.~\cite{Fur} for $T=\mathbb{Z}_+$}]\label{1.3}
Let $(T,X)$ be any semiflow. Then:
\begin{enumerate}
\item An $x\in X$ is called a \textit{$\Delta^*$-recurrent point} if for all neighborhood $U$ of $x$ the set
\begin{gather*}
N_{T^{-1}\circ T}(x,U)=\{\tau\in T^{-1}\circ T\,|\,\tau x\subset U\}
\end{gather*}
is a $\Delta^*$-set of $T$.

\item An $x\in X$ is called an \textit{IP$^*$-recurrent point} of $(T,X)$ if for all neighborhood $U$ of $x$,
$N_{T}(x,U)$
is an IP$^*$-set of $T$, i.e., $N_{T}(x,U)$ intersects non-voidly every IP-set of $T$.
\end{enumerate}

We notice here that when $T$ is group, then $N_{T^{-1}\circ T}(x,U)=N_T(x,U)$. However, in general, it only holds that $N_{T^{-1}\circ T}(x,U)\supset N_T(x,U)$.
\end{sn}

\begin{term}\label{1.4}
Let $x\in X$, $N\subset\langle T^{-1}\circ T\rangle$; then:
\begin{enumerate}
\item For $\varepsilon\in\mathscr{U}_X$ write
\begin{gather*}
C_\varepsilon(N,x)=\{t\in T^{-1}\circ T\,|\,stx\subseteq\varepsilon[sx]\textrm{ and }\varepsilon[stx]\supseteq sx\ \forall s\in N\}.
\end{gather*}
Clearly $C_\varepsilon(\{e\},x)=N_{T^{-1}\circ T}(x,\varepsilon[x])$.

\item When $d\colon X\times X\rightarrow\mathbb{R}$ is a continuous pseudo-metric on $X$ and $\varepsilon>0$, set
\begin{gather*}
B_{\varepsilon,d}[A]=\{y\in X\,|\,\exists a\in A\textit{ s.t. }d(a,y)<\varepsilon\}\intertext{and then write}
C_{\varepsilon,d}(N,x)=\{t\in T^{-1}\circ T\,|\,stx\subset B_{\varepsilon,d}[sx]\textrm{ and }B_{\varepsilon,d}[stx]\supset sx\ \forall s\in N\}.
\end{gather*}
By $d_H(\cdot,\cdot)$ we denote the Hausdorff pseudo-metric induced by $d$ on $X$. Then we have
\begin{gather*}
C_{\varepsilon,d}(N,x)=\{t\in T^{-1}\circ T\,|\,d_H(sx,stx)<\varepsilon\ \forall s\in N\}.
\end{gather*}
\end{enumerate}
\end{term}

It should be noted that the $\Delta^*$-recurrence and $C_\varepsilon(N,x)$ depend on both $(T,X)$ and $(X,T)$, but IP$^*$-recurrence depends only on $(T,X)$.

\begin{lem}\label{1.5}
Let $x\in P_{\!aa}(T,X)$ and $\varepsilon\in\mathscr{U}_X$. Then
$N_{T^{-1}\circ T}(x,\varepsilon[x])$ is a $\Delta^*$-set of $T$. Moreover, if $(T,X)$ is such that $x\mapsto t^{-1}[x]$ is continuous for $t\in T$ and $N\subset\langle T^{-1}\circ T\rangle$ finite, then $C_\varepsilon(N,x)$ is a $\Delta^*$-set of $T$.
\end{lem}

\begin{proof}
Assume for a contradiction that there exists an $\varepsilon\in\mathscr{U}_X$ such that $N_{T^{-1}\circ T}(x,\varepsilon[x])$ is not a $\Delta^*$-set of $T$.
Let $D\subseteq T^{-1}\circ T$ be a $\Delta$-set of $T$ corresponding to a net $\{t_n\}$ in $T$ following Definition~\ref{1.2} such that $D\cap N_{T^{-1}\circ T}(x,\varepsilon[x])=\emptyset$.
Passing to a subnet of $\{t_n\}$ if necessary, let $t_nx\to x^\prime$ and $t_n^{-1}[x^\prime]\to\{x\}$ in $X$ in the sense of Hausdorff topology. Then $t_n^{-1}[x^\prime]\subset\varepsilon[x]$ as $n$ sufficiently big. From this it follows readily that as $m$ sufficiently large, for some $n_m$,
$t_m^{-1}[t_nx]\subseteq\varepsilon[x]$ for all $n\ge n_m$.
This implies $D\cap N_{T^{-1}\circ T}(x,\varepsilon[x])\not=\emptyset$ a contradiction. Thus $N_{T^{-1}\circ T}(x,\varepsilon[x])$ must be a $\Delta^*$-set of $T$.

Finally we can obtain the second part by a slight modification of the above argument with $C_\varepsilon(N,x)$ in place of $N_{T^{-1}\circ T}(x,\varepsilon[x])$. The proof is complete.
\end{proof}

We now consider the almost automorphy from the point of view of recurrence. The following theorem shows that an a.a. point has very strong recurrence.

\begin{thm}[{cf.~\cite[Theorem~9.13]{Fur} for $T=\mathbb{Z}$}]\label{1.6}
A point of $X$ is $\Delta^*$-recurrent for $(T,X)$ if and only if it is an a.a. point of $(T,X)$.
\end{thm}

\begin{proof}
Let $x\in X$ be a $\Delta^*$-recurrent point of $(T,X)$ and suppose that $t_nx\to x^\prime, x_n^{\prime\prime}\to x^{\prime\prime}$ and $t_nx_n^{\prime\prime}=x^\prime$ for some net $\{t_n\}$ in $T$. We need to prove $x=x^{\prime\prime}$. If $x\not=x^{\prime\prime}$, let $V_x$ and $V_{x^{\prime\prime}}$ be two disjoint neighborhoods of $x$ and $x^{\prime\prime}$, respectively. Then $N_{T^{-1}\circ T}(x,V_x)$ is a $\Delta^*$-set of $T$ so that it intersects non-voidly every $\Delta$-set of $T$. Since for $m$ sufficiently big there is an $n_0=n_0(m)$ such that $t_m^{-1}[t_nx]\cap V_{x^{\prime\prime}}\not=\emptyset$ for $n\ge n_0$, hence $A=\{\tau\in T^{-1}\circ T\,|\,\tau x\cap V_{x^{\prime\prime}}\not=\emptyset\}$ is a $\Delta$-set of $T$ by Definition~\ref{1.2}. This is a contradiction to $A\cap N_{T^{-1}\circ T}(x,V_x)=\emptyset$, and so $x=x^{\prime\prime}$.

Conversely, suppose that $x\in P_{\!aa}(T,X)$. Then by Lemma~\ref{1.5}, it follows that $x$ is $\Delta^*$-recurrent for $(T,X)$. This proves Theorem~\ref{1.6}.
\end{proof}

\begin{sn}\label{1.7}
Let $(T,X)$ be any semiflow, which is not necessarily surjective.
\begin{enumerate}
\item[(1)] An $x\in X$ will be called a \textit{distal point} of $(T,X)$ if $x$ is proximal only to itself in $\mathrm{cls}_XTx$.
\item[(2)] If $(T,X)$ is pointwise distal, then $(T,X)$ is said to be \textit{distal}.
\item[(3)] If $x\in X$ is a distal point such that $\textrm{cls}_XTx=X$, then $(T,X)$ is called \textit{point-distal}.
\end{enumerate}
We note here that every distal point is a.p. for all semiflow; see, e.g., \cite{Fur,AD}. Moreover:
\begin{itemize}
\item A point of $X$ is distal if and only if it is IP$^*$-recurrent~\cite[Theorem~4]{DL}; and for $t\in T$, $t^{-1}tx=x$ for every distal point $x$ of a point-distal semiflow $(T,X)$.
\end{itemize}
By
\begin{gather*}
P(X)={\bigcup}_{\varepsilon\in\mathscr{U}_X}{\bigcap}_{t\in T}t^{-1}[\varepsilon]
\end{gather*}
it follows easily that
\begin{itemize}
\item $(T,X)$ is distal if and only if $P(X)=\varDelta_X$.
\end{itemize}
\end{sn}

The following is a consequence of Theorem~\ref{1.6} and (\ref{1.2}b). However, we will present an independent proof here.

\begin{lem}[{cf.~Furstenberg~\cite[Corollary to Theorem~9.13]{Fur} for $T=\mathbb{Z}$}]\label{1.8}
Every $x$ of $P_{\!aa}(T,X)$ is a distal point of $(T,X)$.
\end{lem}

\begin{proof}
Let $x\in P_{\!aa}(T,X)$ and we will show $x$ is an IP$^*$-recurrent point of $(T,X)$. Suppose the contrary that there is a neighborhood $U$ of $x$ such that $N_T(x,U)$ is not an IP$^*$-set in $T$. Then $H:=T\setminus N_T(x,U)$ contains an IP-set of $T$. So there is a sequence $\{s_n\}$ in $T$ such that $s_{n_1}\dotsm s_{n_k}\in H$ for $1\le n_1<n_2<\dotsm<n_k<\infty$ and $1\le k$. Let $t_n=s_1s_2\dotsm s_n$ and we can assume (a subnet of) $t_nx\to z$. Then there is an $m$ such that $t_m^{-1}[t_nx]\subset U$ as $n>m$ sufficiently big. This implies $t_m^{-1}[t_m(s_{m+1}\dotsm s_nx)]\subset U$. Thus $s_{m+1}\dotsm s_nx\in U$, a contradiction to $s_{m+1}\dotsm s_n\in H=T\setminus N_T(x,U)$. Hence $x$ is a distal point of $(T,X)$. The proof is complete.
\end{proof}

\begin{cor}\label{cor1.8A}
If $x\in P_{\!aa}(T,X)$ and $\varepsilon\in\mathscr{U}_X$, then $N_T(x,\varepsilon[x])$ is discretely syndetic in $T$ and so $T\cap N_{T^{-1}\circ T}(x,\varepsilon[x])$ is discretely syndetic in $T$.
\end{cor}

It should be noted here that the ``onto'' condition of $(T,X)$ has played a role in the proof of Lemma~\ref{1.8} and so in Corollary~\ref{cor1.8A}.

Let $N\subset\langle T^{-1}\circ T\rangle$; then $M$ is called a \textit{superset} of $N$ if $N\subseteq M$. The following important result Theorem~\ref{1.10} is a generalization to Veech~\cite[Lemma~2.1.2]{V65} from Bochner a.a. functions on a discrete group to a.a. points of a surjective semiflow on a compact $\textrm{T}_2$ space, which shows that $C_\varepsilon(N,x)$ is a ``big'' subset of $T^{-1}\circ T$ because $T\cap C_\delta(M,x)$ is syndetic in $T$ by Corollary~\ref{cor1.8A}.

\begin{thm}\label{1.10}
Let $(T,X)$ be such that $x\mapsto t^{-1}[x]$ is continuous for all $t\in T$ and $x\in P_{\!aa}(T,X)$. Given $\varepsilon\in\mathscr{U}_X$ and a finite set $N\subset\langle T^{-1}\circ T\rangle$, there exist a $\delta\in\mathscr{U}_X$ and a finite superset $M$ of $N$ such that $\sigma^{-1}\tau\in C_\varepsilon(N,x)$ for all $\sigma,\tau\in T\cap C_\delta(M,x)$.
\end{thm}

\begin{proof}
Let $\Sigma$ be the set of continuous pseudo-metrics on $X$ (cf.~\cite[Theorems~6.19 and 6.29]{Kel}), which generates the uniformity $\mathscr{U}_X$ (cf.~\cite[Theorem~6.15]{Kel}). Assume the contrary; and then for some finite set $N\subset\langle T^{-1}\circ T\rangle$, some $\varepsilon>0$ and some $d\in\Sigma$ and every superset $M$ of $N$ and all $\delta>0$, there must exist $\sigma,\tau\in T\cap C_{\delta,d}(M,x)$ with $\sigma^{-1}\tau\not\in C_{\varepsilon,d}(N,x)$. (Note that given any $\alpha\in\mathscr{U}_X$, there are a $d\in\Sigma$ and an $\epsilon>0$ such that $C_{\epsilon,d}(N,x)\subseteq C_\alpha(N,x)$.)

Given any sequence $\{A_n\}$ of compact subsets of $X$, for simplicity, write $d_H\textrm{-}\lim A_n=K$ for $A_n\to K$ in the sense of the Hausdorff pseudo-metric $d_H$ based on $d$ on $X$.

Choose a sequence of positive numbers $\{\delta_n\}$ decreasing to $0$ so fast that $\sum_{n=1}^\infty\delta_n<\infty$. Let an increasing sequence $\{M_n\}$ of finite supersets of $N$ together with a sequence $(\sigma_n,\tau_n)$ of pairs of elements of $T$ be chosen as follows:

Let $L_1=N\cup\{e\}$ and we set $M_1=L_1\cup L_1^{-1}$. Then $M_1$ is a finite superset of $N$, and so by the assumption there exists $(\sigma,\tau)=(\sigma_1,\tau_1)$ with $\sigma,\tau\in T\cap C_{\delta_1,d}(M_1,x)$ but $\sigma^{-1}\tau\not\in C_{\varepsilon,d}(N,x)$. Having chosen $M_1,\dotsc,M_k$ and $(\sigma_1,\tau_1), \dotsc, (\sigma_k,\tau_k)$ we set $L_{k+1}=M_kM_kN_k$ where $N_k=\{e,\sigma_k,\tau_k,\sigma_k^{-1}\tau_k\}$, and define $M_{k+1}=L_{k+1}\cup L_{k+1}^{-1}$. $M_{k+1}$ is a finite superset of $M_k$, and again by the assumption there exists a pair $(\sigma,\tau)=(\sigma_{k+1},\tau_{k+1})$ such that $\sigma_{k+1},\tau_{k+1}$ both are in $T\cap C_{\delta_{k+1},d}(M_{k+1},x)$ but $\sigma_{k+1}^{-1}\tau_{k+1}\not\in C_{\varepsilon,d}(N,x)$. The construction then proceeds by induction.

Let $T_0=\bigcup_{k=1}^\infty M_k$. If $s,t\in T_0$, then $s,t\in M_k$ and so $t^{-1}\in M_k$ for $k$ sufficiently large implying by construction of $\{M_n\}$ that $st^{-1}\in M_{k+1}\subset T_0$. Thus $T_0$ is a subgroup of $\langle T^{-1}\circ T\rangle$.

A sequence $\{\alpha_k\}$ of elements of $T$ is now defined as follows:
Let $\alpha_1=\tau_1$ and $\alpha_2=\sigma_1$, $\alpha_3=\tau_1\tau_2$ and $\alpha_4=\alpha_1\sigma_2$; for every $k\ge2$ define
\begin{gather*}
\alpha_{2k+1}=\tau_1\tau_2\dotsm\tau_{k+1}\quad \textrm{and}\quad \alpha_{2k+2}=\alpha_{2k-1}\sigma_{k+1}.
\end{gather*}
If $s\in M_k$, then $s\alpha_{2k-1}\in M_{k+1}$, since $s, \tau_1,\dotsc,\tau_{k-1}\in M_k$ and $\tau_k\in N_k$. We have by the triangle inequality of the pseudo-metric $d$
\begin{equation*}
d_H(s\alpha_{2k+1}x,s\alpha_{2k+2}x)\le d_H(s\alpha_{2k-1}\tau_{k+1}x,s\alpha_{2k-1}x)+d_H(s\alpha_{2k-1}x,s\alpha_{2k-1}\sigma_{k+1}x)\le2\delta_{k+1}.
\end{equation*}
Therefore, if we can show that $d_H\textrm{-}\lim\limits_{k\to\infty}s\alpha_{2k+1}x=sy$ exists for $s\in T_0$, then so will $d_H\textrm{-}\lim\limits_{k\to\infty}s\alpha_{k}x=sy$ hold for $s\in T_0$. Now if $s\in M_k$ and $k<j$, then $s\alpha_{2j+1}\in M_{j+2}$. Hence by triangle inequality
\begin{equation*}\begin{split}
d_H(s\alpha_{2k+1}x,s\alpha_{2j+1}x)&\le\sum_{i=0}^{j-k-1}d_H(s\alpha_{2(k+i)+1}x,s\alpha_{2(k+i+1)+1}x)\\
&=\sum_{i=0}^{j-k-1}d_H(s\alpha_{2(k+i)+1}x,s\alpha_{2(k+i)+1}\tau_{k+i+2}x)\\
&\le\sum_{i=0}^{j-k-1}\delta_{k+i+2}
\end{split}\end{equation*}
tends to $0$ as $k\to+\infty$. Therefore by exploiting a subnet of $\{\alpha_{2k+1}x\}$ convergent in the topology of $\mathscr{U}_X$, there exists some $y\in X$ such that
$$
d_H\textrm{-}\lim_{k\to\infty}s\alpha_kx=sy\quad \forall s\in T_0.
$$
Since $x\in P_{\!aa}(T,X)$ and $X$ is compact $\textrm{T}_2$, hence there exist a subnet $\{j\}$ of the sequence $\{k\}$ and a point $y^\prime\in X$ such that
$$
d(y,y^\prime)=0,\quad \lim_js\alpha_{2j}x=sy^\prime\quad \textrm{and} \quad\lim_js\alpha_{2j}^{-1}y^\prime=sx\quad \forall s\in N.
$$
(Since $d$ is only a pseudo-metric on $X$, there is no $y=y^\prime$ in general!) Thus we can choose $k>j$ so big that
\begin{equation}\label{eq1.1}
\max_{s\in N}d_H\left(s\alpha_{2j}^{-1}\alpha_{2k+1}x,sx\right)<\frac{\varepsilon}{2}.
\end{equation}
Now $\alpha_{2j}^{-1}=\sigma_j^{-1}\alpha_{2(j-1)-1}^{-1}=\sigma_j^{-1}\tau_{j-1}^{-1}\dotsm\tau_1^{-1}$ and $\alpha_{2k+1}=\tau_1\dotsm\tau_{k+1}$. Therefore $\sigma_j^{-1}\tau_j\dotsm\tau_{k+1}\in T^{-1}\circ T$ and by $k>j$ we have $\alpha_{2j}^{-1}\alpha_{2k+1}(x)=\sigma_j^{-1}\tau_j\dotsm\tau_{k+1}(x)$ because $(T,X)$ is 1-1 at every a.a. point and $tx\in P_{\!aa}(T,X)$ for $t\in T$. If $s$ belongs to $N$, then
$$
s\sigma_j^{-1}\tau_j\in M_{j+1},\quad s\sigma_j^{-1}\tau_j\tau_{j+1}\in M_{j+2},\quad \dotsc,\quad s\sigma_j^{-1}\tau_j\tau_{j+1}\dotsm\tau_k\in M_{k+1}.
$$
Hence
\begin{equation}\label{eq1.2}
\begin{split}
\max_{s\in N}d_H\left(s\alpha_{2j}^{-1}\alpha_{2k+1}x,s\sigma_j^{-1}\tau_jx\right)&\le\max_{s\in N}\sum_{i=0}^{k-j}d_H\left(s\sigma_j^{-1}\tau_j\dotsm\tau_{j+i}x,s\sigma_j^{-1}\tau_j\dotsm\tau_{j+i+1}x\right)
\le\sum_{i=0}^{k-j}\delta_{j+i+1}.
\end{split}
\end{equation}
Let $j$ be chosen so large that $\sum_{n\ge j}\delta_n<\frac{\varepsilon}{2}$ which is possible since $\sum_{n=1}^\infty\delta_n<\infty$. Then by (\ref{eq1.1}), (\ref{eq1.2}) and the triangle inequality of $d_H$,
$$
d_H\left(s\sigma_j^{-1}\tau_jx,sx\right)<\varepsilon\quad \forall s\in N,
$$
and this inequality contradicts that $(\sigma_j,\tau_j)$ was so chosen that $\sigma_j^{-1}\tau_j\not\in C_{\varepsilon,d}(N,x)$. This thus completes the proof of Theorem~\ref{1.10}.
\end{proof}

\begin{cor}\label{1.11}
Let $(T,X)$ be a flow with phase group $T$ and $x\in P_{\!aa}(T,X)$. Given $\varepsilon\in\mathscr{U}_X$, a finite set $N\subset T$ and an integer $n>0$, there exists a finite superset $M$ of $N$ and a $\delta\in\mathscr{U}_X$ such that if $\tau_1,\dotsc,\tau_n\in C_\delta(M,x)$, then $\tau_1^{\epsilon_1}\dotsm\tau_n^{\epsilon_n}\in C_\varepsilon(N,x)$ for all choice $\epsilon_i=0,+1$ or $-1$.
\end{cor}

\begin{proof}
The proof is by induction on $n$. For $n=1$ select a set $M\supset N$ and $\delta\in\mathscr{U}_X$ by Theorem~\ref{1.10}. If $\tau\in C_\delta(M,x)$, then $\tau^{\epsilon_1}\in C_\varepsilon(N,x)$ for $\epsilon_1=0,1$ or $-1$, since $e\in C_\delta(M,x)$ and $\tau^{-1}e\in C_\varepsilon(N,x)$.

Suppose the corollary holds for some integer $n\ge 1$, and let $M_1\supset N$ and $\delta_1\in\mathscr{U}_X$ be chosen by Theorem~\ref{1.10} so that whenever $\tau,\sigma\in C_{\delta_1}(M_1,x)$, then $\sigma^{-1}\tau\in C_\varepsilon(N,x)$.

Using Theorem~\ref{1.10} once more let $M_2\supset M_1$ and $\delta_2\in\mathscr{U}_X$ be chosen when $\tau,\sigma\in C_{\delta_2}(M_2,x)$, then $\sigma^{-1}\tau\in C_{\delta_1}(M_1,x)$. Note that if $\tau\in C_{\delta_2}(M_2,x)$, then $\tau^{-1}\in C_{\delta_1}(M_1,x)$. By our induction assumption we choose a set $M\supset M_2$ and $\delta\in\mathscr{U}_X$ such that if $\tau_1,\dotsc,\tau_n\in C_\delta(M,x)$, then $\tau_1^{\epsilon_1}\dotsm\tau_n^{\epsilon_n}\in C_{\delta_2}(M_2,x)$ for all choice $\epsilon_j=0,1$ or $-1$. Let $\tau_1,\dotsc,\tau_{n+1}$ be elements of $C_\delta(M,x)$, and suppose $\gamma=\tau_1^{\epsilon_1}\dotsm\tau_{n+1}^{\epsilon_{n+1}}=\gamma^\prime\gamma^{\prime\prime}$ where $\gamma^\prime=\tau_1^{\epsilon_1}\dotsm\tau_n^{\epsilon_n}$ and $\gamma^{\prime}=\tau_{n+1}^{\epsilon_{n+1}}$ with again $\epsilon_j=0,1$ or $-1$. Then both $\gamma^\prime\in C_{\delta_2}(M_2,x)$ and $\gamma^{\prime\prime}\in C_{\delta_2}(M_2,x)$. By our choice of $M_2$ and $\delta_2$, it follows easily that $(\gamma^\prime)^{-1}\in C_{\delta_1}(M_1,x)$, $\gamma^{\prime\prime}\in C_{\delta_1}(M_1,x)$, and finally $\gamma=\left((\gamma^\prime)^{-1}\right)^{-1}\gamma^{\prime\prime}\in C_\varepsilon(N,x)$. The corollary then follows by induction.
\end{proof}

Therefore, when $(T,X)$ is a flow and $x\in P_{\!aa}(T,X)$, then for every neighborhood $U$ of $x$ and $n\ge2$ there are discretely syndetic sets $A_1,\dotsc,A_n$ in $T$ such that $A_1\dotsm A_nx\subset U$.
\section{Bohr almost automorphy of semiflows}\label{sec2}
In this section we will mainly consider almost automorphy from Bohr's viewpoint of recurrence. First we will introduce the basic notion\,---\,Bohr a.a. point of surjective semiflows. The principal results of the discussion are Theorems \ref{thm2.2} and \ref{thm2.5}. Theorem~\ref{thm2.2} asserts that the equivalence of discrete Bohr almost automorphy with Bochner almost automorphy. This also follows that a discretely periodic point of any semiflow is an a.a. point. Particularly, Theorem~\ref{thm2.5} says that an a.a. point is an l.a.p. point if our phase semigroup $T$ is a group or if $T$ is an abelian semigroup.
\begin{sh}
Let $(T,X)$ be a surjective semiflow in this section unless specified otherwise.
\end{sh}
\subsection{Bohr a.a. vs Bochner a.a.}
\begin{sn}\label{def2.1}
Let $(T,X)$ be any semiflow not necessarily surjective. Then:
\begin{enumerate}
\item An $x\in X$ shall be called a \textit{Bohr a.a. point} of $(T,X)$ if for all $\varepsilon\in\mathscr{U}_X$, there is a subset $B=B_T(x,\varepsilon)$ of $T$ such that:
\begin{enumerate}
\item[i)] $B$ is syndetic in $T$.
\item[ii)] If $t_1,t_2\in B$, then $t_2^{-1}[t_1x]\subset\varepsilon[x]$; i.e., $B^{-1}Bx\subseteq\varepsilon[x]$.
\end{enumerate}
\item If here $B$ is discretely syndetic in $T$, then $x$ will be called a \textit{discrete Bohr a.a. point} of $(T,X)$.
\end{enumerate}
\end{sn}

See \cite[Definition~2.1.2]{V65} for Bohr a.a. functions on a discrete group, which requires in addition $B^{-1}=B$. Property ii) is a kind of ``two syndetic sets'' condition. However, even if $T$ is a topological group, $B^{-1}$ need not be syndetic in $T$ in the sense of Definition~\ref{sn0.2}.

\begin{thm}\label{thm2.2}
Let $(T,X)$ be such that $x\mapsto t^{-1}[x]$ is continuous for $t\in T$ and let $x_0\in X$. Then:
\begin{enumerate}
\item[$(1)$] The $x_0$ is a discrete Bohr a.a. point if and only if  $x_0\in P_{\!aa}(T,X)$.

\item[$(2)$] If $(T,X)$ is a flow, then $x_0\in P_{\!aa}(T,X)$ iff for all $\varepsilon\in\mathscr{U}_X$ there is a subset $B=B_T(x_0,\varepsilon)$ of $T$ such that
\begin{enumerate}
\item[$\mathrm{i)}$] $B$ is discretely syndetic in $T$,
\item[$\mathrm{ii)}$] $B^{-1}Bx_0\subseteq\varepsilon[x_0]$,
\item[$\mathrm{iii)}$] $B=B^{-1}$.
\end{enumerate}
\end{enumerate}
\end{thm}

\begin{note}
Property (2) implies that every regular a.p. point of a flow is an a.a. point. Here $x_0$ is called a regular a.p. point if and only if $N_T(x_0,U)$ contains a syndetic subgroup of $T$ for all neighborhood $U$ of $x_0$.
\end{note}

\begin{proof}
(1): Let $x_0$ be a discrete Bohr a.a. point of $(T,X)$ and $\{t_n\}$ a net in $T$. Since $X$ is compact $\textrm{T}_2$, there is no loss of generality in assuming $t_nx_0\to y$, $x_n^\prime\to x_0^\prime$ and $t_nx_n^\prime=y$. To prove that $x_0\in P_{\!aa}(T,X)$, it is sufficient to show $x_0^\prime=x_0$. Assume the contrary. Then there is some $\varepsilon\in\mathscr{U}_X$ such that $(x_0,x_0^\prime)\not\in\varepsilon\circ\varepsilon\circ\varepsilon$. Now choose a set $B=B_T(x_0,\varepsilon)$ satisfying Definition~\ref{def2.1}. Since $B$ is syndetic in $T$, there exist elements $s_1,\dotsc,s_m$ of $T$ such that each $t\in T$ may be written $s_jt=\tau$ where $\tau\in B$ and $1\le j\le m$. For each $t_n$ we can write $s_jt_n=\tau_n$ where $j=j(n)$. There are but finitely many $s_j$, so there will exist a subnet $\{\beta_k\}$ of $\{t_n\}$ such that $s_{j_0}\beta_k=\tau_k$ where $j_0$ is independent of $k$. It remains true for $\{\beta_k\}$ that $\beta_kx_0\to y$ and $\beta_kx_k^\prime=y$. Then let $k$ then $i>k$ be chosen so large that
\begin{gather*}
\varepsilon[x_0^\prime]\cap\beta_k^{-1}[\beta_ix_0]\not=\emptyset\quad \textrm{and}\quad  \beta_k^{-1}[\beta_ix_0]\subseteq\beta_k^{-1}s_{j_0}^{-1}[s_{j_0}\beta_ix_0]=\tau_k^{-1}[\tau_ix_0].
\end{gather*}
Further by condition ii) of Definition~\ref{def2.1}, we can conclude that $(x_0,x_0^\prime)\in\varepsilon\circ\varepsilon\circ\varepsilon$, which is a contradiction. Therefore $x_0\in P_{\!aa}(T,X)$ as was to be proved.

Conversely suppose that $x_0\in P_{\!aa}(T,X)$. By Definition~\ref{sn0.9}, $x_0\in P_{\!aa}(T,X)$ in the sense of the discrete $T$. Given $\varepsilon\in\mathscr{U}_X$ choose a finite superset $M$ of $N=\{e\}$ and $\delta\in\mathscr{U}_X$ as in Theorem~\ref{1.10}. Define $B=T\cap C_\delta(M,x_0)$ which is discretely syndetic in $T$ by Lemma~\ref{1.5}. If $t_1,t_2\in B$ then $t_2^{-1}t_1\in C_\varepsilon(\{e\},x_0)$ using Theorem~\ref{1.10}. Thus $x_0$ enjoys properties i) and ii) of Definition~\ref{def2.1}. Thus $x_0$ is discrete Bohr a.a. point of $(T,X)$.

(2): We only need to show the ``only if'' part because of (1). For this, let $x_0\in P_{\!aa}(T,X)$. Applying Corollary~\ref{1.11} with $N=\{e\}$ and $n=2$ and then having set $B=C_\delta(M,x_0)\cup (C_\delta(M,x_0))^{-1}$, $B$ satisfies conditions i), ii) and iii). Thus (2) of Theorem~\ref{thm2.2} holds.
The proof is complete.
\end{proof}

\begin{cor}[Pushing-out of almost automorphy]
$\,$
\begin{enumerate}
\item Let $\pi\colon(T,X)\rightarrow(T,Y)$ be an epimorphism of two invertible semiflows. If $x\in P_{\!aa}(T,X)$, then $\pi(x)\in P_{\!aa}(T,Y)$.
\item Every invertible factor of an a.a. semiflow is also an a.a. semiflow.
\end{enumerate}
\end{cor}

\begin{que}[{Lifting of a.a. points; cf.~Sell-Shen-Yi~1998~\cite{SSY}}]
Let $\pi\colon(T,X)\rightarrow(T,Y)$ be an epimorphism of minimal flows and $N$ an integer with $N\ge2$. There are two open questions on a.a. points, which are of general interests.
\begin{enumerate}
\item {\it If $\pi$ is of almost $N$-to-1 type and $(T,Y)$ is equicontinuous, when is $(T,X)$ an a.a. flow?}
\item {\it If $\pi$ is of $N$-to-1 type and $(T,Y)$ is an a.a. flow, when is $(T,X)$ an a.a. flow?}
\end{enumerate}
According to Ellis' ``two-circle'' minimal set, to obtain positive solutions to these questions, we need to add conditions on the phase group $T$ or the phase space $X$.
\end{que}

Recall following Definition~\ref{sn0.3} that $x$ is l.a.p. for $(T,X)$ if and only if for every neighborhood $U$ of $x$ there is a syndetic set $A$ of $T$ and a neighborhood $V$ of $x$ such that $AV\subset U$.

Based on Veech's structure theorem any a.a. point is l.a.p. in the flows. However, by independent approaches, the following theorem asserts that an a.a. point is an l.a.p. point in many important situations.

\begin{thm}\label{thm2.5}
Let $(T,X)$ be minimal invertible such that $T$ is a group or an abelian semigroup. If $x\in P_{\!aa}(T,X)$, then $x$ is an l.a.p. distal point of $(T,X)$ in the sense of discrete $T$.
\end{thm}
\begin{proof}
Let $x\in P_{\!aa}(T,X)$. Since the almost automorphy still holds under the discrete topology of $T$, we may suppose $T$ is discrete. First by Theorem~\ref{thm2.2}, $x$ is a Bohr a.a. point of $(T,X)$. Now let $U$ be a closed neighborhood of $x$. We need to find a neighborhood $V$ of $x$ and a syndetic set $A$ of $T$ with the property $AV\subseteq U$.

$(a)$: Let $T$ be an abelian semigroup. Then by Definition~\ref{def2.1}, there is a discretely syndetic set $B$ of $T$ such that
$B^{-1}Bx\subset U$ and so $BB^{-1}x\subset U$.
Take a finite subset $K=\{k_1,\dotsc,k_m\}$ of $T$ with $T=K^{-1}B$ (i.e. $\forall t\in T, \exists k\in K$ s.t. $kt\in A$). Therefore $T^{-1}=B^{-1}K=KB^{-1}$.
Since $X=\mathrm{cls}_XTx=\mathrm{cls}_XT^{-1}x$ (cf.~\cite{AD}), $X=\bigcup_{i=1}^mk_i\mathrm{cls}_XB^{-1}x$. Thus $W:=\mathrm{Int}_X\mathrm{cls}_XB^{-1}x\not=\emptyset$. Choose some $s\in T$ with $sx\in W$ and then $x\in s^{-1}W$. Having set $V=s^{-1}W$ and $A=sB$, it holds that
$$
AV=Bss^{-1}W=BW\subseteq B\mathrm{cls}_XB^{-1}x\subseteq\mathrm{cls}_XBB^{-1}x\subseteq U.
$$
Since $A$ is syndetic in $T$ and $U$ is arbitrary, thus $x$ is an l.a.p. point of $(T,X)$.

$(b)$: Let $T$ be a group. Then by (2) of Theorem~\ref{thm2.2}, we can find a discretely syndetic set $B$ of $T$ such that $BBx\subset U$. This implies that $x$ is an l.a.p. point of $(T,X)$ by an argument similar to that of the case $(a)$. Indeed, take a finite set $K\subset T$ with $T=KB$ for $B$ is syndetic. Then $W=\textrm{Int}_XBx\not=\emptyset$ so that $V=s^{-1}W$ is a neighborhood of $x$ for some $s\in T$. Thus for $A=Bs$, $AV=BW\subseteq \textrm{cls}_XBBx\subseteq U$.

Finally, since an a.a. point is always a distal point of $(T,X)$ by Lemma~\ref{1.8}, the proof of Theorem~\ref{thm2.5} is therefore completed.
\end{proof}

Theorem~\ref{thm2.5} is is a generalization to \cite[Theorem~19]{AM} which is for abelian group acting on compact metric spaces by different approaches.

In addition it should be noted that if here $(T,X)$ is not minimal, then this statement need not be true. In fact, every fixed point is an a.a. point but not necessarily an l.a.p. one when $(T,X)$ is not a minimal semiflow. For example, let $T=\mathbb{R}$ and $X=\mathbb{R}\cup\{*\}$ the one-point compactification of $\mathbb{R}$ where $*$ at infinity. Define $(t,x)\mapsto t+x$ then $*$ is an a.a. but not an l.a.p. point for the flow $(T,X)$.

\begin{rem}
In \cite{MW} McMahon and Wu proved that $x\in X$ is an l.a.p. point of a minimal flow iff for every neighborhood $U$ of $x$ there are discretely syndetic sets $A,B$ in $T$ such that $ABx\subset U$. Thus Theorem~\ref{thm2.5} in flows may also follow from Corollary~\ref{1.11}.
\end{rem}

\begin{que}\label{que2.7}
Is an a.a. point of any (invertible) minimal semiflow with nonabelian phase semigroup (e.g. amenable semigroup or almost right C-semigroup) an l.a.p. point?
\end{que}

\begin{que}\label{q2.8}
Let $(T,X)$ be a flow with non-discrete phase group $T$. Does it hold that $x$ is a Bohr a.a. point if and only if $x$ is a discrete Bohr a.a. point?
\end{que}

Let $T$ be a topological group and by $\mathfrak{N}_e$ we denote the system of neighborhoods of $e$ in $T$. In two special cases we can obtain a positive solution to Question~\ref{q2.8} as follows:

\begin{thm}\label{thm2.9}
Let $(T,X)$ be a flow with non-discrete phase group $T$. If $T$ is such that given $U\in\mathfrak{N}_e$ and $B\subset T$ there exists $V\in\mathfrak{N}_e$ with $BV\subseteq UB$, then $x$ is a Bohr a.a. point if and only if $x$ is a discrete Bohr a.a. point.
\end{thm}

\begin{proof}
Assume $x$ is a Bohr a.a. point of $(T,X)$ and $\varepsilon\in\mathscr{U}_X$. Then there exists a compact neighborhood $W$ of $x$ with $W\subset\varepsilon[x]$ and a syndetic subset $B$ in $T$ such that $B^{-1}Bx\subseteq W$. Further there is some $U\in\mathfrak{N}_e$ with $UB^{-1}Bx\subset\varepsilon[x]$. Thus, there is $V\in\mathfrak{N}_e$ such that $B^{-1}VBx\subset\varepsilon[x]$. Moreover, we can take an $N\in\mathfrak{N}_e$ with $N^{-1}N\subseteq V$, so $(NB)^{-1}(NB)x\subseteq\varepsilon[x]$.
Note that $NB$ is discretely syndetic in $T$. Indeed, let $K$ be the compact subset of $T$ such that $T=KB$. Then there is a finite subset $F=\{k_1,\dotsc, k_n\}$ of $K$ with $FN\supseteq K$ so that $T=F(NB)$.
Thus $x$ is a discrete Bohr a.a. point of $(T,X)$.
Since the other direction is obvious, so the proof of Theorem~\ref{thm2.9} is completed.
\end{proof}

\begin{thm}\label{thm2.10}
Let $(T,X)$ be a flow with non-discrete phase group $T$. If $T$ is such that given $U\in\mathfrak{N}_e$ and $B\subset T$ there exists $V\in\mathfrak{N}_e$ with $VB\subseteq BU$, then $x$ is a Bohr a.a. point if and only if $x$ is a discrete Bohr a.a. point.
\end{thm}

\begin{proof}
First we will show that $\textrm{cls}_T(B^{-1}B)=\bigcap_{N\in\mathfrak{N}_e}(NB)^{-1}NB$ for $B\subset T$.
Indeed, since for each $V\in\mathfrak{N}_e$ there is some $W\in\mathfrak{N}_e$ with $W^{-1}W\subseteq V$ and $W^{-1}W\in\mathfrak{N}_e$, thus $\bigcap_{W\in\mathfrak{N}_e}(WB)^{-1}WB=\bigcap_{V\in\mathfrak{N}_e}B^{-1}VB$. Moreover, since $B^{-1}B\subset B^{-1}VB$ for all $V\in\mathfrak{N}_e$, $\textrm{cls}_T(B^{-1}B)\subseteq B^{-1}VB$ so that $\textrm{cls}_T(B^{-1}B)\subseteq \bigcap_{V\in\mathfrak{N}_e}B^{-1}VB$. On the other hand, $\textrm{cls}_T(B^{-1}B)= \bigcap_{V\in\mathfrak{N}_e}B^{-1}BV\supseteq\bigcap_{W\in\mathfrak{N}_e}B^{-1}WB$. Thus $\textrm{cls}_T(B^{-1}B)=\bigcap_{W\in\mathfrak{N}_e}(WB)^{-1}WB$.

Let $x$ be a Bohr a.a. point of $(T,X)$ and $\varepsilon\in\mathscr{U}_X$.
Now by $\textrm{cls}_X{B^{-1}B}x\subseteq\textrm{cls}_X{B^{-1}Bx}$ for all $B\subseteq T$, there exist a syndetic set $B$ in $T$ and an $N\in\mathfrak{N}_e$ such that $(NB)^{-1}(NB)x\subseteq\varepsilon[x]$. Since $NB$ is discretely syndetic in $T$, $x$ is a discrete Bohr a.a. point of $(T,X)$.
The other direction is obvious, so the proof of Theorem~\ref{thm2.10} is completed.
\end{proof}

\subsection{Bohr-Veech points}\label{sec2.2}
\begin{sn}\label{sn2.11}
Let $(T,X)$ be any semiflow not necessarily surjective.
\begin{enumerate}
\item A subset $A$ of $T$ is said to be \textit{relatively dense} in $T$ if there is a finite set $F\subseteq T$ such that $Ft\cap A\not=\emptyset$ and $tF\cap A\not=\emptyset$ for all $t\in T$.
\item An $x\in X$ is called a \textit{Bohr-Veech point} of $(T,X)$ if given $\varepsilon\in\mathscr{U}_X$ there is a subset $B$ of $T$ such that:
\begin{enumerate}
\item $B$ is relatively dense in $T$;
\item $B^{-1}Bx\subseteq\varepsilon[x]$.
\end{enumerate}
\end{enumerate}
Clearly, any Bohr-Veech point of an invertible semiflow is an a.a. point by Theorem~\ref{thm2.2}. When $T$ is a discrete abelian semigroup, any Bohr a.a. point is a Bohr-Veech point in the invertible case.
\end{sn}

Given any $\tau\in T$, let $L_\tau\colon t\mapsto\tau t$ be the left translation of $T$. Since $T$ is only a semigroup, $L_\tau T\subsetneqq T$ in general. Thus for $B\subset T$, $L_\tau^{-1}B$ is possibly an empty subset of $T$.

Recall the reflection $(X,T)$ has the phase mapping $(x,t)\mapsto xt=t^{-1}x$ by Definition~\ref{sn0.1}. Then the following theorem is a weak solution to Question~\ref{que2.7}.

\begin{thm}\label{thm2.12}
Let $(T,X)$ be minimal invertible with $x\in X$. If $x$ is a Bohr-Veech point of $(T,X)$, then $x$ is an l.a.p. point of $(X,T)$.
\end{thm}
\begin{note}
It should be interesting to know whether or not $x$ is an l.a.p. of $(T,X)$ in the setting of Theorem~\ref{thm2.12}.
\end{note}
\begin{proof}
Given $\varepsilon,\alpha\in\mathscr{U}_X$ with $\alpha\circ\alpha\circ\alpha\subset\varepsilon$, let $B$ be a relatively dense set in $T$ such that $B^{-1}Bx\subset\alpha[x]$ following Definition~\ref{sn2.11}. Since $\textrm{Int}_X\textrm{cls}_X{Bx}\not=\emptyset$ by a standard argument, we can take some $\tau\in T$ and some $\delta\in\mathscr{U}_X$ such that $B^{-1}\tau(\delta[x])\subset\varepsilon[x]$; that is, $(\tau^{-1}B)^{-1}(\delta[x])\subset\varepsilon[x]$. Let $B^\prime=L_\tau^{-1}B$, which is a subset of $T$ such that $(\delta[x])B^\prime\subseteq\varepsilon[x]$. To prove the theorem, it is sufficient to show that $B^\prime$ is left syndetic in $T$. Indeed, since $B$ is relatively dense in $T$, there is a finite subset $F$ of $T$ with $B\cap tF\not=\emptyset$ for all $t\in T$. Let $t\in T$ be any given and then set $t^\prime=\tau t$. Since $B\cap t^\prime F\not=\emptyset$, there are $b\in B$ and $f\in F$ with $t^\prime f=b$ so that $\tau tf=b$; i.e., $tf\in L_\tau^{-1}(b)\subset B^\prime$. This shows that $B^\prime\cap tF\not=\emptyset$ for all $t\in T$.
So $B^\prime$ is left syndetic in $T$ and thus $x$ is an l.a.p. point of $(X,T)$. This then proves Theorem~\ref{thm2.12}.
\end{proof}
\section{Veech's structure theorem for a.a. flows}\label{sec3}
This section will be mainly devoted to proving Veech's Structure Theorem for a.a. flows by using Theorem~\ref{thm2.5}. First we will need an important result due to Ellis and Gottschalk.

\begin{lem}[{cf.~\cite[Proposition~5.27]{E69}}]\label{lem3.1}
If $(T,X)$ is an l.a.p. flow with the phase group $T$, then $P(X)=Q(X)$ and moreover $(T,X/P(X))$ is an equicontinuous flow.
\end{lem}

In fact, $(T,X/P(X))$ is actually the maximal equicontinuous factor of an l.a.p. flow $(T,X)$. Lemma~\ref{lem3.1} is a corollary of the following Lemma~\ref{lem3.2}.

Recall from Definition~\ref{0.11} that $(x,y)\in Q^-(T,X)$ if and only if there are nets $x_n\to z, y_n\to z$ and $\{t_n\}$ in $T$ such that $t_nx_n\to x$ and $t_ny_n\to y$. Thus if $(T,X)$ is invertible, then $Q^-(T,X)=Q(X,T)$ where $(X,T)$ is the reflection of $(T,X)$ (cf.~Definition~\ref{sn0.1}). Then Lemma~\ref{lem3.1} may be generalized as follows:

\begin{lem}\label{lem3.2}
If $(T,X)$ is an l.a.p. semiflow, then $P(T,X)\supseteq Q^-(T,X)$ and so $P[x]\supseteq Q^{-}[x]$ for all $x\in X$. $($Hence if $T$ is a group $P(T,X)=Q(T,X)$.$)$
\end{lem}

\begin{proof}
Let $(x,y)\in Q^-(T,X)$; then there are nets $\{t_n\}$ in $T$, $\{x_n,y_n\}$ in $X\times X$, and $z\in X$ such that
$(x_n,y_n)\to(z,z)$, $t_nx_n\to x$, and $t_ny_n\to y$.
Since $z$ is an l.a.p. point of $(T,X)$, hence for every $\varepsilon\in\mathscr{U}_X$ there are some $\delta\in\mathscr{U}_X$ and a syndetic set $A$ in $T$ such that
$A(\delta[z])\subseteq\frac{\varepsilon}{3}[z]$, where $\frac{\varepsilon}{3}\in\mathscr{U}_X$ such that $\frac{\varepsilon}{3}\circ\frac{\varepsilon}{3}\circ\frac{\varepsilon}{3}\subseteq\varepsilon$. Then there is a compact subset $K$ of $T$ such that for all $n$ there are $k_n\in K$ and $a_n\in A$ with $k_nt_n=a_n$. By passing to a subnet of $\{t_n\}$ if necessary, we may suppose that
$$
k_n\to k\in T\quad \textrm{and}\quad k_nt_nx_n=a_nx_n\to x^\prime\in\varepsilon[z],\ k_nt_ny_n=a_ny_n\to y^\prime\in\varepsilon[z].
$$
Thus $k(x,y)=(x^\prime,y^\prime)\in\varepsilon$ and $Q^-(T,X)\subseteq P(T,X)$. The proof is complete.
\end{proof}

\begin{lem}[\cite{DX, AD}]\label{lem3.3}
A semiflow $(T,X)$ is a.p. if and only if it is equicontinuous surjective if and only if $Q(T,X)=\varDelta_X$.
\end{lem}

\begin{lem}[\cite{AD}]\label{lem3.4}
If $(T,X)$ is equicontinuous surjective, then it is distal.
\end{lem}

\begin{thm}[{cf.~\cite{G56} or \cite[Corollary~5.28]{E69} for $T$ a group}]\label{thm3.5}
Let $(T,X)$ be any semiflow. Then $(T,X)$ is a.p. if and only if $(T,X)$ is l.a.p. and distal.
\end{thm}

\begin{proof}
Assume $(T,X)$ is a.p.; then $(T,X)$ is l.a.p. and moreover it is equicontinuous surjective by Lemma~\ref{lem3.3}. So $(T,X)$ is distal by Lemma~\ref{lem3.4}.

Conversely, suppose $(T,X)$ is l.a.p. distal and so $(T,X)$ is invertible with $P(T,X)=\varDelta_X$. By Lemma~\ref{lem3.2}, it follows that $Q(X,T)=\varDelta_X$ and so $(X,T)$ and then $(T,X)$ is a.p. by Lemma~\ref{lem3.3}.
\end{proof}

Next using the foregoing preparations we can readily prove Veech's Structure Theorem for a.a. flows as follows.
\begin{0.13}
If $(T,X)$ is an a.a. flow, then $(T,X)$ is an almost 1-1 extension of an equicontinuous flow $(T,Y)$ via $\pi\colon(T,X)\rightarrow(T,Y)$ such that $P_{\!aa}(T,X)=\{x\in X\,|\,\pi^{-1}[\pi(x)]=\{x\}\}$.
\end{0.13}

\begin{proof}
Let $(T,X)$ be an a.a. flow and then it is minimal by Lemma~\ref{1.8}. By Theorem~\ref{thm2.5} and Lemma~\ref{lem0.6}, $(T,X)$ is an l.a.p. minimal flow. Let $Y=X/P(X)$ and $(T,X)\xrightarrow{\pi}(T,Y)$ the canonical projection. Then $(T,Y)$ is an equicontinuous factor of $(T,X)$ by Lemma~\ref{lem3.1}. Now by Lemma~\ref{1.8}, $P[x]=\{x\}$ for all $x\in P_{\!aa}(T,X)$ so that $\pi$ is of almost 1-1 type such that $\pi^{-1}(\pi(x))=\{x\}$.

Finally, let $x\in X$ such that $\pi^{-1}(\pi(x))=\{x\}$. If $t_nx\to x^\prime$ and $t_n^{-1}x^\prime\to x^{\prime\prime}$ for a net $\{t_n\}$ in $T$, then by $t_n\pi(x)\to\pi(x^\prime)$ and $t_n^{-1}\pi(x^\prime)\to\pi(x^{\prime\prime})$, we can see that $\pi(x)=\pi(x^{\prime\prime})$ because $Q(Y)=\varDelta_Y$. Thus $x=x^{\prime\prime}$ and then $x\in P_{\!aa}(T,X)$. Thus $P_{\!aa}(T,X)=\{x\in X\,|\,\pi^{-1}(\pi(x))=\{x\}\}$.
The proof is thus complete.
\end{proof}

In fact, as far as we have known, this is the first complete proof of Veech's structure theorem of a.a. flows. The converse of Theorem~\ref{0.13} holds as follows:

\begin{thm}\label{thm3.6}
Let $(T,X)$ and $(T,Y)$ be two minimal semiflows. If $\pi\colon(T,X)\rightarrow(T,Y)$ is of almost 1-1 type and $(T,Y)$ is equicontinuous surjective, then $(T,X)$ is an a.a. semiflow such that $\{x\,|\,\pi^{-1}[\pi(x)]=\{x\}\}\subseteq P_{\!aa}(T,X)$.
\end{thm}

\begin{proof}[Proof I]
Let $x_0\in X$ be such that $\pi^{-1}[\pi(x_0)]=\{x_0\}$. Let $S_\textit{eq}$ be the equicontinuous structure of $(T,X)$; that is, $S_\textit{eq}$ is the minimal invariant closed equivalence relation $R$ on $X$ such that $(T,X/R)$ is equicontinuous. Then writing $[x]_\textit{eq}=S_\textit{eq}[x]$ for all $x\in X$, we have canonical epimorphisms
\begin{gather*}
\pi\colon (T,X)\xrightarrow[{x\mapsto[x]_{\textit{eq}}}]{\pi_\textit{eq}}(T,X/S_\textit{eq})\xrightarrow[{[x]_{\textit{eq}}\mapsto\pi(x)}]{\theta}(T,Y).
\end{gather*}
This implies that $\pi_\textit{eq}^{-1}[[x_0]_\textit{eq}]=\pi_\textit{eq}^{-1}[\pi_\textit{eq}(x_0)]=\{x_0\}$. Since $\pi_\textit{eq}(Q[x_0])\subseteq Q[[x_0]_\textit{eq}]=P[[x_0]_\textit{eq}]=\{[x_0]_\textit{eq}\}$ (noting $(T,Y)$ is distal by Lemma~\ref{lem3.4} and $\theta^{-1}[\theta([x_0]_\textit{eq})]=[x_0]_\textit{eq}$), $Q[x_0]=\{x_0\}$ and thus $x_0\in P_{\!aa}(T,X)$. The proof is complete.
\end{proof}

\begin{proof}[Proof II] (When $x\mapsto t^{-1}[x]$ is continuous for $t\in T$.)
At first, by Lemma~\ref{lem3.3}, $(T,Y)$ is a pointwise a.a. semiflow. Since $\pi\colon X\rightarrow Y$ is of almost 1-1 type, there exists a point $x\in X$ such that $\pi^{-1}(y)=\{x\}$. Further, for every open neighborhood $U$ of $x$ there is an open neighborhood $V$ of $y$ with $\pi^{-1}(V)\subseteq U$. Thus if $B\subset T$ such that $B^{-1}By\subseteq V$, then $B^{-1}Bx\subseteq U$. Then Theorem~\ref{thm3.6} follows easily from (1) of Theorem~\ref{thm2.2}. The proof is complete.
\end{proof}

The following consequence is due to Veech~\cite{V77}, which has been independently proved in \cite{AGN} by using different approaches.

\begin{cor}[Veech~\cite{V65,V77}]\label{cor3.7}
If $(T,X)$ is a minimal flow, then $(T,X)$ is equicontinuous if and only if $P_{\!aa}(T,X)=X$.
\end{cor}

\begin{proof}
If $(T,X)$ is equicontinuous, then it follows that $Q(X)=\varDelta_X$ and then we can conclude $P_{\!aa}(T,X)=X$. Conversely, let $\pi\colon(T,X)\rightarrow(T,Y)$ be an almost 1-1 extension of an equicontinuous flow $(T,Y)$ given by Theorem~\ref{0.13}. If $P_{\!aa}(T,X)=X$, then $\pi$ is an isomorphism so that $(T,X)$ is equicontinuous.
\end{proof}

In fact, the above important statement still holds for minimal surjective semiflows; see Theorem~\ref{thm4.5} in $\S\ref{sec4}$. However, the minimality is crucial in the proof of Corollary~\ref{cor3.7}; see Example~\ref{exa4.7} for a counterexample.

\begin{thm}\label{thm3.8}
Let $(T,X)$ be invertible point-distal. If $(T,X)$ and its reflection $(X,T)$ both are l.a.p., then $(\langle T\rangle,X)$ is an a.a. flow.
\end{thm}

\begin{proof}
By Lemma~\ref{lem3.2}, $P(T,X)\supseteq Q^-(T,X)=Q(X,T)\supseteq P(X,T)\supseteq Q(T,X)$. Since $tQ(X,T)\subseteq Q(X,T)$ for $t\in T$, $P(T,X)=Q(T,X)$ is closed invariant. Further by a standard argument (cf., e.g., \cite[Corollary~6.11, p.~88]{Aus}), we can see $Q(X)$ is an equivalence relation on $X$. This implies that
$\pi\colon(T,X)\rightarrow(T,X/Q(X))$ is an almost 1-1 extension such that $(T,X/Q(X))$ is minimal equicontinuous invertible. Let $\langle T\rangle_X=\langle T\rangle$ be associated to $(T,X)$. Then $\pi\colon X\rightarrow X/Q(X)$ may be naturally extended as an almost 1-1 extension $\pi\colon(\langle T\rangle_X,X)\rightarrow(\langle T\rangle_X,X/Q(X))$. Noting that $(\langle T\rangle_X,X/Q(X))$ is a minimal equicontinuous flow (cf.~\cite[Theorem~1.15]{AD}), it follows from Theorem~\ref{thm3.6} that $(\langle T\rangle,X)$ is an a.a. flow. The proof is thus complete.
\end{proof}

We note that if ``$(X,T)$ is l.a.p.'' is an implication of that $(T,X)$ is l.a.p. in the above theorem, then we can get a positive solution to (2) of Question~\ref{0.18}.
\section{Equicontinuity and pointwise almost automorphy}\label{sec4}
Veech proved that a function on a discrete group is a.p. if and only if it is a.a. and each of its limit points is a.a. (cf.~\cite[Theorem~3.3.1]{V65}). Now we shall show that Veech's theorem still holds for general surjective semiflows.
The principal result Theorem~\ref{thm4.5} claims that a minimal surjective semiflow is equicontinuous if and only if it is pointwise a.a., which actually generalizes Veech's theorem (Corollary~\ref{cor3.7} in $\S\ref{sec3}$).

\subsection{Pointwise almost automorphy}
By Veech's structure theorem (Theorem~\ref{0.13}), we can easily see that a pointwise a.a. minimal flow is a.p. (Corollary~\ref{cor3.7}). In fact, this also follows easily from Theorem~\ref{thm3.5}. However, for an invertible minimal semiflow, there is no such a Veech structure theorem at hands and moreover, we do not know if an a.a. point must be an l.a.p. point.
In view of these reasons, we need new ideas for proving that ``a minimal pointwise a.a. semiflow is a.p.''

\begin{lem}\label{lem4.1}
Let $(T,X)$ be a surjective semiflow. If $(T,X)$ is equicontinuous, then $P_{\!aa}(T,X)=X$.
\end{lem}

\begin{proof}[Proof I]
Since $(T,X)$ is equicontinuous surjective, then by \cite{AD} it follows that $(\langle T\rangle,X)$ is equicontinuous and hence $P_{\!aa}(\langle T\rangle,X)=X$. This evidently shows that $P_{\!aa}(T,X)=X$. The proof is complete.
\end{proof}

\begin{proof}[Proof II]
This follows from Theorem~\ref{thm3.6} with $(T,Y)=(T,X)$ and $\pi=\textit{id}_X$.
\end{proof}

\begin{sn}[\cite{E69,Fur,Aus,AD}]\label{def4.2}
Let $E(X)$ be the closure of $T$ in $X^X$ under the pointwise topology, where each $t$ of $T$ is identified with the transition map $x\mapsto tx$ of $X$ to itself corresponding to $(T,X)$. Then $E(X)$ is called the \textit{Ellis enveloping semigroup} of $(T,X)$.
\end{sn}

In addition we will need two lemmas for any semiflow $(T,X)$.

\begin{lem}[{cf.~\cite[Lemma~5.2]{DX}}]\label{lem4.3}
Let $(T,X)$ be a surjective semiflow. Then $P_{\!aa}(T,X)=X$ if and only if $E(X)$ is a compact $\textrm{T}_2$ topological group with $e=\textit{id}_X$.
\end{lem}

\begin{lem}[{cf.~\cite[Proposition~5.5]{DX}}]\label{llem4.4}
$E(X)$ is a compact $\textrm{T}_2$ topological group with $e=\textit{id}_X$ if and only if $(T,X)$ is distal and $(T,\mathrm{cls}_XTx)$ is equicontinuous for each $x\in X$.
\end{lem}

\begin{thm}\label{thm4.5}
Let $(T,X)$ be a minimal surjective semiflow; then $P_{\!aa}(T,X)=X$ iff $(T,X)$ is equicontinuous invertible iff $(T,X)$ is an a.p. semiflow.
\end{thm}

\begin{proof}
If $(T,X)$ is equicontinuous invertible, then by Lemma~\ref{lem4.1} it follows that $P_{\!aa}(T,X)=X$. Conversely, if $P_{\!aa}(T,X)=X$, then $E(X)$ is a compact $\textrm{T}_2$ topological group by Lemma~\ref{lem4.3}. Further, since $(T,X)$ is minimal, hence by Lemma~\ref{llem4.4} $(T,X)$ is equicontinuous invertible. The proof is complete.
\end{proof}

\begin{cor}
Let $(T,X)$ be a minimal surjective semiflow. Then $(T,X)$ is equicontinuous invertible if and only if each point of $X$ is $\Delta^*$-recurrent for $(T,X)$.
\end{cor}

\begin{proof}
This follows easily from the consequences of Theorem~\ref{thm4.5} and Theorem~\ref{1.6}.
\end{proof}

It is interesting to find a proof for Theorem~\ref{thm4.5} without using Ellis enveloping semigroup.
Moreover, it should be noted that if $(T,X)$ is not minimal, then the consequence of Theorem~\ref{thm4.5} need not be true. Let us see a simple classical example.

\begin{exa}\label{exa4.7}
Let $\varphi$ be the self homeomorphism of the disk $\mathbb{D}$ in the plane such that $\varphi(r,\theta)=(r,\theta+r)$, where $(r,\theta)$ are polar coordinates of $\mathbb{D}$. Clearly, $(\varphi,\mathbb{D})$, corresponding to a flow on $\mathbb{D}$ with phase group $\mathbb{Z}$, is pointwise a.a., i.e.,
$P_{\!aa}(\varphi,\mathbb{D})=\mathbb{D}$; but it is not equicontinuous.
\end{exa}

Let $(T,X)$ be surjective and $x\in X$. We say $(T^{-1}\circ T,X)$ is equicontinuous at $x$ if given $\varepsilon\in\mathscr{U}_X$ there is a $\delta\in\mathscr{U}_X$ such that $t^{-1}[sy]\subset\varepsilon[t^{-1}[sx]]$ and $t^{-1}[sx]\subset\varepsilon[t^{-1}[sy]]$ for all $y\in\delta[x]$ and $s,t\in T$.
The following theorem generalizes Veech's \cite[v) of Theorem~1.2.1]{V65} that is for a.a. functions on discrete groups.

\begin{thm}\label{thm4.8}
Let $(T,X)$ be a surjective semiflow such that $x\mapsto t^{-1}[x]$ is continuous for $t\in T$ and let $x_0\in X$. If $P_{\!aa}(T,X)\ni x_k\to x_0$ in $X$ and $(T^{-1}\circ T,X)$ is equicontinuous at $x_0$, then $x_0\in P_{\!aa}(T,X)$.
\end{thm}

\begin{proof}
Let $(T,X)$ be invertible.
Let $t_nx_0\to x^\prime, x^\prime t_n\to x^{\prime\prime}$ for a net $\{t_n\}$ in $T$. We need to show $x_0=x^{\prime\prime}$. For this, let $\delta,\varepsilon\in\mathscr{U}_X$ be any given with $\delta\subset\varepsilon$ such that $(T^{-1}\circ T)\delta[x_0]\subseteq\varepsilon$. Since $x_k\to x_0$, then $(x_k,x_0)\in\delta$ as $k\ge k_0$. Let a subnet $\{s_i\}$ of $\{t_n\}$ be so chosen that $\lim_is_i^{-1}\lim_is_ix_k=x_k$. Noting that $x^{\prime\prime}=\lim_is_i^{-1}\lim_is_ix_0$ we have that
$$\left(x^{\prime\prime},{\lim}_is_i^{-1}{\lim}_is_ix_k\right)=\left({\lim}_is_i^{-1}{\lim}_is_ix_0,{\lim}_is_i^{-1}{\lim}_is_ix_k\right)\in\varepsilon.$$
By the triangle inequality,
$$\left(x^{\prime\prime},x_0\right)\in\left({\lim}_is_i^{-1}{\lim}_is_ix_0,{\lim}_is_i^{-1}{\lim}_is_ix_k\right)\left({\lim}_is_i^{-1}{\lim}_is_ix_k,x_0\right)\in\varepsilon\circ\varepsilon.$$
Since $\varepsilon$ be arbitrary, thus $x_0=x^{\prime\prime}$. The general case may be similarly proved and we omit the details here. The proof is complete.
\end{proof}
\subsection{Almost periodic and a.a. functions of discrete groups}
\begin{sht}
In this subsection we suppose that
\begin{enumerate}
\item $G$ is a discrete group and by $L^\infty=L^\infty(G)$ we denote the set of real-valued bounded functions on $G$.

\item The pointwise topology on $L^\infty$ is the topology for which $f_n\to f\Leftrightarrow f_n(x)\to f(x)\ \forall x\in G$. For any $T\subseteq G$, $H_T(f)=\textrm{cls}\{tf\,|\,t\in T\}$ under the pointwise topology.

\item The uniform topology on $L^\infty$ is the topology for which $f_n\to f\Leftrightarrow f_n(x)\to f(x)$ uniformly for $x\in G$.
\end{enumerate}
\end{sht}
\noindent
Under each of the pointwise and uniform topologies, we can then define the right-translate flow $(G,L^\infty)$ as follows:
$$\tau_R\colon G\times L^\infty\rightarrow L^\infty,\quad (t,f)\mapsto tf,$$
which is isometric (i.e., $\|tf\|=\|f\|$ for all $t\in G$ and $f\in L^\infty$) and so which is equicontinuous with respect to the uniform topology on $L^\infty$.

\begin{sn}[{cf.~\cite{Neu, GH, V65, E69} for $T=G$}]\label{4.9}
Let $T$ be a subgroup of $G$.
\begin{enumerate}
\item $f\in L^\infty$ is said to be (right) $T$-a.p. in the sense of Bochner--von Neumann if and only if $Tf$ is relatively compact in $(L^\infty,\|\cdot\|)$.


\item $f\in L^\infty$ is called an (right) $T$-\textit{a.a.} function on $G$ if and only if $t_nf\to \eta, f_n^\prime\to f^\prime$ and $t_nf_n^\prime=\eta$ implies $f=f^\prime$ for every net $\{t_n\}$ in $T$ in the pointwise topology on $L^\infty$.
\end{enumerate}
\end{sn}

%

Now as a consequence of Theorem~\ref{thm4.5}, we can concisely obtain the following classical result which is due to Bochner and Veech.

\begin{thm}[{\cite{B62} and \cite[Theorem~3.3.1]{V65} for $T=G$}]\label{thm4.10}
Let $T$ be a syndetic subgroup of $G$ and $\xi\in L^\infty$. Then $\xi$ is Bochner--von Neumann $T$-a.p. if and only if every element of $H_T(\xi)$ is an $T$-a.a. function on $G$.
\end{thm}

\begin{proof}
Necessity: Let $\xi$ be $T$-a.p. in the sense of Bochner--von Neumann. Let $F=\mathrm{cls}_{\|\cdot\|}T\xi$ be the uniform closure of $T\xi$. Then $(T,F)$ is a minimal equicontinuous flow and so it is pointwise $T$-a.a. by Theorem~\ref{thm4.5} under the uniform topology. If $s_n\xi\to f$, for a net $\{s_n\}$ in $T$, in the pointwise topology, then $f\in F$. Thus $F=H_T(\xi)$. Given $f\in F$, assume $t_nf\to y$ and $t_n^{-1}y\to f^\prime$ in the sense of the pointwise topology. Since $t_nf\in F$, then by passing to a subsequence $\{\alpha_i\}$ of the net $\{t_n\}$, it holds that $\alpha_if\to y$ and $\alpha_i^{-1}y\to f$ in the sense of the uniform topology. Thus $f=f^\prime$. This shows every $f$ of $F$ is an $T$-a.a. function on $G$.

Sufficiency: Assume every element of $H_T(\xi)$ is an $T$-a.a. function on $G$. Since $\xi$ is bounded, $H_T(\xi)$ is compact $\textrm{T}_2$ under the pointwise topology. Moreover, by Lemmas~\ref{lem4.3} and \ref{llem4.4}, $(T,H_T(\xi))$ is minimal. Thus $(T,H_T(\xi))$ is a minimal equicontinuous flow under the pointwise topology by Theorem~\ref{thm4.5} again. Further the pointwise topology coincides with the uniform topology on $H_T(\xi)$ so $\xi$ is an $T$-a.p. function. In fact, let $G=KT$ for some finite set $K\subset G$ and assume $f_n\to f$ in $H_T(\xi)$ in the sense of pointwise topology and let $\epsilon>0$. Then there is some $\delta>0$ such that whenever $z,w\in H_T(\xi)$ such that $\max_{k\in K}|z(k)-w(k)|<\delta$ then $\max_{k\in K}|tz(k)-tw(k)|<\epsilon$ for all $t\in T$ by equicontinuity of $T$ $\tau_R$-acting on $H_T(\xi)$. This means that
as $n>n_0$ for some $n_0$, $\max_{k\in K}|tf_n(k)-tf(k)|<\epsilon$ for all $t\in T$. Thus
$$\|f_n-f\|={\sup}_{x\in G}|f_n(x)-f(x)|={\sup}_{t\in T}{\max}_{k\in K}|f_n(kt)-f(kt)|\le\epsilon.$$
This implies that $f_n\to f$ in the sense of uniform topology. Thus $T\xi$ is relatively compact under the uniform topology and $\xi$ is $T$-a.p. in the sense of Bochner--von Neumann.
\end{proof}

It should be mentioned that if $T$ was only a semigroup not a group in Theorem~\ref{thm4.10}, $(T,H_T(\xi))$ would not be surjective so that Theorem~\ref{thm4.5} would play no role in the foregoing proof.

We can of course define the left $T$-a.p. and $T$-a.a. functions on $G$. In 1965 \cite{V65} for the case $T=G$, Veech's proof needs to utilize the left almost automorphy of $\xi$ associated to the left-translate flow $\tau_L\colon L^\infty\times G\rightarrow L^\infty$ by $(f,t)\mapsto ft$ where $ft(x)=f(tx)$ for all $x\in G$. Also see \cite[Theorem~2.3.1]{V65} for the special case $G=(\mathbb{Z},+)$. Then the following results motivate our Questions~\ref{sec0.1}(i) and \ref{sec0.1}(ii).
\begin{description}
\item[(V)] A bounded function on $G$ is left a.a. if and only if it is right a.a. (cf.~Veech~\cite[Theorem~1.3.1]{V65} or $\S$\ref{sec0.1}.(i)).
\item[(L)] A bounded function on $G$ is left a.p. in the sense of Bochner--von Neumann if and only if it is right a.p. in the sense of Bochner--von Neumann (cf.~von Neumann~\cite{Neu} or Loomis~\cite[Lemma~41B]{L}; also see $\S$\ref{sec0.1}.(ii)).
\end{description}

It is known that on a complete metric space not locally compact, the orbit of an a.p. point need not be relatively compact; see \cite[Proposition~4.6]{CD}. However, for functions on $G$ we can obtain the following simple observation:
\begin{itemize}
\item Let $f\in L^\infty$ and $T$ a subgroup of $G$. Then $f$ is Bochner-von Neumann $T$-a.p. if and only if it is a.p. for $T$ in the sense of Definition~\ref{sn0.2}.2, i.e., given $\varepsilon>0$, $N_T(f,\varepsilon)=\{t\in T\,|\,\|tf-f\|<\varepsilon\}$ is right-syndetic in $T$.
\end{itemize}
\section{Veech relationships of invertible semiflows}\label{sec5}
Using Veech's relations we will prove Theorem~\ref{0.14} and Theorem~\ref{0.17} in this section and consider invertible semiflows with abelian phase semigroups.
The Veech relations are not only useful for almost automorphy but also for capturing the equicontinuous structure of flows \cite{V68}. Here we will consider two kinds of relations introduced by Veech in \cite{V65,V68}.

\subsection{Relations $V$ and $D$}
\begin{sn}[{cf.~\cite{V65,AGN} for $T$ a group}]\label{def5.1}
Let $(T,X)$ be any semiflow. Then:
\begin{enumerate}
\item We say that an ordered pair $(x,x^\prime)$ is in ``Veech'' relation, denoted $(x,x^\prime)\in V(T,X)$ or $x^\prime\in V[x]$, if there exist a $y\in X$, nets $\{t_n\}$ in $T$ and $\{x_n^\prime\}$ in $X$ such that $t_nx\to y$, $x_n^\prime\to x^\prime$ and $t_nx_n^\prime=y$.

\item We will say that $x$ is \textit{well-proximal to} $x^\prime$ if there is a net $\{t_n\}$ in $T$ such that $t_n(x,x^\prime)\to (x^\prime,x^\prime)$.
\end{enumerate}

If $(T,X)$ is surjective and $x\in P_{\!aa}(T,X)$, then $x$ is only well-proximal itself in $X$.
If $T$ is a group or an abelian semigroup, then $V(T,X)$ is invariant. It is obvious that $V[x]=\{x\}$ if and only if $x\in P_{\!aa}(T,X)$. Moreover, $V(T,X)\subseteq Q(T,X)$. However, even if $(T,X)$ is a flow, here we cannot in general show that $V(T,X)$ is an equivalence relation. In fact, $V(T,X)$ is not symmetric in general.
\end{sn}

The following theorem in the special case of $T=\mathbb{Z}$ is just Furstenberg's \cite[Proposition~9.14]{Fur}.

\begin{thm}\label{thm5.2}
Let $X$ be a compact metric space, $(T,X)$ a surjective semiflow such that $x\mapsto t^{-1}[x]$ is continuous for $t\in T$, and $x\in X$. If $x$ is well-proximal to $x^\prime$, then $x^\prime\in V[x]$.
\end{thm}

\begin{proof}
Since $x$ is well-proximal to $x^\prime$ and $X\times X$ is metrizable, there is a sequence $\{p_n\}$ in $T$ with $p_{m+1}\dotsm p_{m+\ell}x\to x^\prime$ as $m\to\infty$ for all $\ell\ge1$, following Furstenberg's construction (cf., e.g.,~\cite[Theorem~2.17]{Fur}). Set $s_n=p_1p_2\dotsm p_n$ for $n=1,2,\dotsc$ and suppose $s_nx\to y$ and $s_n^{-1}[y]\to K$ for some closed subset $K$ (passing to a subsequence if necessary). We claim $x^\prime\in K$. Otherwise, let $U_{x^\prime}$ and $U_K$ be two disjoint neighborhoods of $x^\prime$ and $K$, respectively. Then $s_m^{-1}[s_nx]\subset U_K$ for $n>m$ as $m$ sufficiently big. This contradicts that $s_m^{-1}[s_{m+\ell}x]\supseteq p_{m+1}\dotsm p_{m+\ell}x\to x^\prime$ as $m\to\infty$. This proves Theorem~\ref{thm5.2}.
\end{proof}

Since each point of $X$ is well proximal to some a.p. point of a semiflow (\cite{Fur}), therefore we can conclude that if $x$ is an a.a. point in Theorem~\ref{thm5.2}, then $x$ is an a.p. point by a line different with Lemma~\ref{1.8}.

\begin{cor}\label{cor5.3}
Let $(T,X)$ be a surjective semiflow on a compact metric space $X$ such that $x\mapsto t^{-1}[x]$ is continuous for $t\in T$. Then:
\begin{enumerate}
\item[$(1)$] If $(x,y)\in P(X)$ with $y$ an a.p. point, then $(x,y)\in V(T,X)$. Particularly $P(T,X)\subseteq V(T,X)$ if $(T,X)$ is minimal.
\item[$(2)$] If $(x,y)\in P(T,X)$, then $V[x]\cap V[y]$ contains an a.p. point of $(T,X)$.
\end{enumerate}
\end{cor}

\begin{proof}
(1): Let $E(X)$ be the Ellis enveloping semigroup as in Definition~\ref{def4.2}. If $(x,y)\in P(T,X)$, then there is a minimal left ideal $I$ in $E(X)$ such that $p(x)=p(y)$ for all $p\in I$. Let $J=\{u\in I\,|\,u^2=u\}$.
Since $y$ is an a.p. point of $(T,X)$, $u(x)=u(y)=y$ for some $u\in J$ so that $x$ is well proximal well to $y$. Thus (1) of Corollary~\ref{cor5.3} holds by Theorem~\ref{thm5.2}.

(2): Given $u\in J$, let $z=u(x)=u(y)$ and so $u(z)=z$. Then by Theorem~\ref{thm5.2}, it follows that $z\in V[x]\cap V[y]$ and $z$ is an a.p. point of $(T,X)$. Thus we have proved (2) of Corollary~\ref{cor5.3}.
\end{proof}

\begin{note}
The metric on $X$ plays a role in our proof of Theorem~\ref{thm5.2}. It would be interested to know if this statement holds for flows with compact $\textrm{T}_2$ non-metrizable phase spaces.
\end{note}

The second Veech relation we will consider in this section is the following one.

\begin{sn}[{cf.~\cite{V68} for $T$ a group}]\label{def5.4}
Let $(T,X)$ be any semiflow. Given $x\in X$, define $D[x]$ for $(T,X)$ by
$$
D[x]={\bigcap}_{\varepsilon\in\mathscr{U}_X}\mathrm{cls}_X\left\{\bigcup t^{-1}[sx]\,|\, s,t\in N_T(x,\varepsilon[x])\right\}.
$$
$D[x]$ is closed, and of course $x\in D[x]$. Then we will say $(x,y)\in D(T,X)$ if and only if $y\in D[x]$. Clearly, $D(T,X)$ is invariant when $T$ is a group or an abelian semigroup.
\end{sn}

By the definition, it is easy to verify that
\begin{itemize}
\item {\it $y\in D[x]$ $\Leftrightarrow$ $\exists$ $\{t_n\}, \{s_n\}$ in $T$ and $\{y_n\}$ in $X$ s.t. $t_nx\to x, s_nx\to x, y_n\to y$ and $t_ny_n=s_nx$.}
\end{itemize}

To show $V[x]\subseteq D[x]$ for all point $x\in X$ of any invertible minimal semiflow $(T,X)$, we will need an algebraic lemma.

\begin{lem}\label{lem5.5}
Let $T$ be a semigroup of surjective self-maps of $X$ and $x\in X$. If $A$ is a discretely syndetic set of $T$, then $A^{-1}A$ is a $\Delta^*$-set of $T$ in the sense that for any net $\{t_i^\prime\}$ in $T$, there is a subnet $\{t_n\}$ of $\{t_i^\prime\}$ such that $t_m^{-1}[t_nx]\subseteq A^{-1}Ax$ for all $m<n$.
\end{lem}

\begin{proof}
Let $\{t_n\}$ be any net in $T$. Since $A$ is syndetic, there is a finite subset $K=\{k_1,\dotsc,k_m\}$ of $T$ such that $T=K^{-1}A$. Then for $n$, $k_jt_n=a_n$ for some $a_n\in A$ and some $j$ with $1\le j\le m$. By considering a subnet of $\{t_n\}$ if necessary, we can assume $kt_n=a_n$ and then $t_n^{-1}k^{-1}=a_n^{-1}$ for all $n$, where $k\in K$ is independent of $n$. Thus by $a_m^{-1}ky=t_m^{-1}k^{-1}ky\supseteq t_m^{-1}y$ for $y\in X$, it follows that
$$
t_m^{-1}[t_nx]\subseteq a_m^{-1}[kt_nx]=a_m^{-1}[a_nx]\subseteq A^{-1}Ax
$$
for all $m<n$. This shows that $A^{-1}A$ is a $\Delta^*$-set of $T$.
\end{proof}

\begin{thm}[{cf.~\cite[Theorem~16]{AGN} for $T$ a group}]\label{thm5.6}
Let $(T,X)$ be a surjective semiflow such that $x\mapsto t^{-1}[x]$ is continuous for $t\in T$. If $(T,X)$ is minimal and $x\in X$, then:
\begin{enumerate}
\item[$(1)$] $V[x]\subseteq D[x]$.
\item[$(2)$] $D[x]=U[x]\subseteq Q[x]$, where $U[x]=\{y\,|\,\exists y_n\to y\textrm{ and }\{t_n\}\textrm{ in }T\textrm{ s.t. }t_n(x,y_n)\to (x,x)\}$.
\end{enumerate}
\end{thm}

\begin{proof}
(1). Let $x\in X$ and $y\in V[x]$. Then there is a net $\{t_n\}$ in $T$ such that $y\in\lim_m\lim_nt_m^{-1}[t_nx]$. Let $\varepsilon\in\mathscr{U}_X$; then $A=N_T(x,\varepsilon[x])$ is discretely syndetic in $T$. Thus $A^{-1}A$ is a $\Delta^*$-set of $T$ by Lemma~\ref{lem5.5}. Obviously this shows that $y\in\textrm{cls}_XA^{-1}Ax$ and thus $y\in D[x]$.

(2). Let $y\in D[x]$, then there are nets $\{t_n\}, \{s_n\}$ in $T$ and $\{y_n\}$ in $X$ with $t_nx\to x$, $s_nx\to x$, $y_n\to y$, and $t_ny_n=s_nx$ by Definition~\ref{def5.4}. Then $t_n(x,y_n)=(t_nx,s_nx)\to (x,x)$. Thus $y\in U[x]$. On the other hand, $U[x]\subseteq D[x]$ follows easily from the minimality of $(T,X)$. Finally $U[x]\subseteq Q[x]$ is evident.
The proof is complete.
\end{proof}
\subsection{Two applications of Veech's relations}
\begin{0.14}
Let $(T,X)$ be a minimal flow. Then $(T,X)$ is an a.a. flow if and only if there exists a point $x\in X$ such that $Q[x]=\{x\}$.
\end{0.14}

\begin{proof}
Let $Q[x]=\{x\}$ for some point $x$ of $X$. Then $V[x]=\{x\}$ by Theorem~\ref{thm5.6} so that $x\in P_{\!aa}(T,X)$. Since $(T,X)$ is minimal, thus $(T,X)$ is an a.a. flow.

Conversely, assume $(T,X)$ is an a.a. flow. By Theorem~\ref{thm2.5} or by Theorem~\ref{0.13}, $(T,X)$ is an l.a.p. flow so $P(X)=Q(X)$ by Lemma~\ref{lem3.1}. Moreover, by Lemma~\ref{1.8}, $P[x]=\{x\}$ for some $x\in X$. Thus $Q[x]=\{x\}$. This proves Theorem~\ref{0.14}.
\end{proof}

Recall that $(X,T)$, the reflection of $(T,X)$, is defined by $(x,t)\mapsto xt=t^{-1}x$ (see Definition~\ref{sn0.1}). The following lemma is a generalization of \cite[(2) of Theorem~5.31 for right C-semigroups]{AD}.

\begin{lem}\label{lem5.7}
Let $(T,X)$ be invertible. Assume $T$ is an almost right C-semigroup. If $(T,X)$ is minimal, then $(X,T)$ is also minimal.
\end{lem}

\begin{proof}
First of all, if $T$ is compact, then the lemma is obviously true. Indeed, let $(T,X)$ be minimal and then for all $x,y\in X$, $Ty=\textrm{cls}_X{Ty}=X$ and so $ty=x$ for some $t\in T$. This implies that for all $x,y\in X$, $y=t^{-1}x=xt$ for some $t\in T$ and thus $\textrm{cls}_X{xT}=xT=X$ for every $x\in X$. Hence $(X,T)$ is minimal.

Now suppose $T$ is non-compact and that $(T,X)$ is minimal.
Since $G:=\{t\,|\,\textrm{cls}_T(T\setminus Tt)\textrm{ is compact in }T\}$ is dense in $T$, $\textrm{cls}_XTx=\textrm{cls}_XGx$ for all $x\in X$.

We show $(X,T)$ is minimal. For this, for $x\in X$, define the \textit{$\alpha$-limit set} of $x$ w.r.t. $(T,X)$ by
$\alpha_T(x)={\bigcap}_{F\in\mathscr{F}}\textrm{cls}_XxF^c$,
where $\mathscr{K}$ is the collection of compact subsets of $T$ and $F^c$ is the complement of $F$ in $T$. Clearly, $\alpha_T(x)$ is closed non-empty by the ``finite intersection property'' and $\alpha_T(x)\subseteq\textrm{cls}_XxT$.

We will show that $\alpha_T(x)$ is an invariant set of $(T,X)$. For this, let $y\in\alpha_T(x)$ and $s\in G$. Let $F\in\mathscr{K}$. Since $K:=Fs\cup\textrm{cls}_T(T\setminus Ts)$ is compact and $F^cs\supseteq K^c$ (for $K^c\subseteq Ts=(F\cup F^c)s$), then $y\in\textrm{cls}_X xF^cs$ so there is a net $\{t_n\}$ in $F^c$ such that
$xt_ns\to y$ and $xt_n\to z$.
Thus $zs=y$ and $z\in\textrm{cls}_XxF^c$. This shows that $sy=ys^{-1}\in\alpha_T(x)$. Thus $G\alpha_T(x)\subseteq\alpha_T(x)$ and then $T\alpha_T(x)\subseteq\alpha_T(x)$, i.e., $\alpha_T(x)$ is an invariant closed set of $(T,X)$.

However, since $(T,X)$ is minimal by hypothesis, $\alpha_T(x)=X$ for all $x\in X$. Thus $\textrm{cls}_XxT=X$ for all $x\in X$ and so $(X,T)$ is minimal. The proof is complete.
\end{proof}

\begin{0.17}
Let $(T,X)$ be an invertible minimal semiflow with $T$ an almost right C-semigroup and $x_0\in X$. Then the following two statements hold:
\begin{enumerate}
\item[$(a)$] $(X,T)$ is minimal.
\item[$(b)$] If $x_0$ is a distal l.a.p. point of $(T,X)$, then $x_0$ is an a.a. point of $(X,T)$.
\end{enumerate}
\end{0.17}

\begin{proof}
First of all, by Lemma~\ref{lem5.7}, $(X,T)$ is minimal.
Next by Lemma~\ref{lem0.6}, $(T,X)$ is l.a.p. and then it follows from Lemma~\ref{lem3.2} that $Q(X,T)\subseteq P(T,X)$. Since $x_0$ is a distal point of $(T,X)$, thus, for $(X,T)$, $Q[x_0]=\{x_0\}$. Then by Theorem~\ref{thm5.6}, it follows that $x_0$ is an a.a. point of $(X,T)$. The proof is complete.
\end{proof}

\subsection{Reflection principles and structure theorem of abelian semigroups}
Here the following Theorem~\ref{thm5.8} actually gives us three reflection principles. Here the ``abelian'' condition will be needed for (2) $\Rightarrow$ (3) or (b) $\Rightarrow$ (c) and (2) $\Leftrightarrow$ (b).

\begin{thm}\label{thm5.8}
Let $(T,X)$ be minimal invertible with $T$ abelian and $x\in X$. Then the followings are pairwise equivalent.
\begin{enumerate}
\item[$(1)$] $Q_{(T,X)}[x]=\{x\}$.
\item[$(\mathrm{a})$]$Q_{(X,T)}[x]=\{x\}$.
\item[$(2)$]$x\in P_{\!aa}(T,X)$.
\item[$(\mathrm{b})$]$x\in P_{\!aa}(X,T)$.
\item[$(3)$]$x$ is a distal and l.a.p. point of $(T,X)$.
\item[$(\mathrm{c})$]$x$ is a distal and l.a.p. point of $(X,T)$.
\end{enumerate}
Here $Q_{(T,X)}[x]$ and $Q_{(X,T)}[x]$ stand for the cells of $Q(T,X)$ and $Q(X,T)$ at $x$, respectively.
\end{thm}

\begin{proof}$\,$

(1) $\Rightarrow$ (2) and (a) $\Rightarrow$ (b). This follows immediately from Theorem~\ref{thm5.6}.

(2) $\Rightarrow$ (3) and (b) $\Rightarrow$ (c). This follows immediately from Theorem~\ref{thm2.5}.

(2) $\Leftrightarrow$ (b). This follows at once from (1) of Theorem~\ref{thm2.2} and the fact that if $B$ is left syndetic in $T$ then it is also right syndetic in $T$.

(c) $\Rightarrow$ (1). By Lemma~\ref{lem0.6}, $(X,T)$ is l.a.p. with $x$ a distal point. Moreover by Lemma~\ref{lem3.2}, $P(X,T)\supseteq Q(T,X)$.
Further, since $x$ is a distal point of $(X,T)$, thus $Q_{(T,X)}[x]=\{x\}$ and so condition (1) holds.

Therefore, it holds that
\begin{equation*}
(1) \Rightarrow (2) \Leftrightarrow (\mathrm{b}) \Rightarrow (\mathrm{c}) \Rightarrow (1),
\end{equation*}
and symmetrically we can show that
\begin{equation*}
(\mathrm{a}) \Rightarrow (\mathrm{b}) \Leftrightarrow (2) \Rightarrow (3) \Rightarrow (\mathrm{a}).
\end{equation*}
This completes the proof of Theorem~\ref{thm5.8}.
\end{proof}

In fact, Veech's structure theorem also holds for $T$ in abelian semigroup as follows, which implies Theorem~\ref{thm4.5} when $T$ is abelian.

\begin{thm}\label{thm5.9}
Let $(T,X)$ be any semiflow. Then
\begin{enumerate}
\item[$(1)$] If $(T,X)$ is minimal with $T$ abelian such that $x\mapsto t^{-1}[x]$ is continuous for $t\in T$, then $(T,X)$ is a.a. if and only if $(T,X)$ is an almost 1-1 extension of an equicontinuous invertible semiflow.
\item[$(2)$] Let $(T,X)$ be an almost 1-1 extension of a minimal semiflow $(T,Y)$ via $\pi\colon X\rightarrow Y$. If $(T,X)$ is point-distal, then $\pi$ is 1-1 at every distal point of $(T,X)$.
\end{enumerate}
\end{thm}

\begin{proof}
(1): The ``if'' part is evident by Theorem~\ref{thm3.6}. So we only show the ``only if'' part and we can assume by Lemma~\ref{1.1} that $(T,X)$ is an a.a. minimal invertible semiflow. By Theorem~\ref{thm5.8}, $(T,X)$ and $(X,T)$ both are l.a.p. semiflows. Thus by Lemma~\ref{lem3.2}, it follows that
$$
P(X,T)\supseteq Q(T,X)\supseteq P(T,X)\supseteq Q(X,T).
$$
Whence $P(T,X)=Q(T,X)$ is a closed invariant equivalent relation on $X$. This implies that $\pi\colon (T,X)\rightarrow(T,X/Q(X))$ is of almost $1$-$1$ type such that $(T,X/Q(X))$ is an equicontinuous invertible semiflow.

(2): Let $y_0\in Y$ such that $\pi^{-1}[y_0]=\{x_0\}$. Let $x$ be a distal point of $(T,X)$ and set $y=\pi(x)$. Take $x_1,x_2\in\pi^{-1}[y]$. Since $(T,Y)$ is minimal and $z\mapsto\pi^{-1}[z]$ is upper-semi continuous, there is a net $\{t_n\}$ in $T$ such that $t_n(x_1,x_2)\to(x_0,x_0)$. Then by distality of $(T,X)$ at $x$, $x_1=x_2$. Thus $\pi$ is 1-1 at $x$; that is, $\pi^{-1}[\pi(x)]=\{x\}$. The proof is complete.
\end{proof}
\subsection{The density of $V(T,X)$ in $D(T,X)$}
Next we will show that $V(T,X)$ is dense in $D(T,X)$ for all minimal invertible semiflow $(T,X)$. If $T$ is an abelian group, this was proved by Veech~(cf.~\cite[Theorem~1.2]{V68}).

We first consider a special case as follows. Let $G=\langle T\rangle$ and let $X=H_T(\xi)$ which is the pointwise closure of the $T$-translates of a given bounded vector-valued function $\xi$ (valued in the $d$-dimensional complex-space $\mathbb{C}^d$) on $G$ and $(T,X)$ be as in $\S$\ref{sec0.1}.3. Moreover, $(T,X)$ is 1-1 here; i.e., $x\mapsto tx$ is 1-1 for $t\in T$. Then:

\begin{lem}\label{lem5.10}
Let $(T,X)$ be minimal where $X=H_T(\xi)$. Then:
\begin{enumerate}
\item[$(1)$] Given $x\in X, x^\prime\in D[x]$, and a finite set $S\subseteq T$, there is a net $\{\beta_i\}$ in $T$ such that $\lim_i\beta_ix$ and $\lim_j\lim_i\beta_j^{-1}\beta_ix=x^{\prime\prime}$ exist with $x^{\prime\prime}(s)=x^\prime(s)$ for all $s\in S$. Hence $V[x]$ is dense in $D[x]$ for all $x\in X$.
\item[$(2)$] If $T$ is countable, then $V[x]=D[x]$ for all $x\in X$.
\end{enumerate}
\end{lem}

\begin{proof}
(1). First, for $\delta>0$ and $z,w\in X$, we shall say that $(z,w)\in\delta$ or $\rho(z,w)<\delta$ if and only if $\|z(s)-w(s)\|<\delta$ for all $s\in S$.
Let $\{\delta_n\}_{n=1}^\infty$ be a sequence of positive numbers such that $\sum_n\delta_n<\infty$.
If $x^\prime\in D[x]$, then by Definition~\ref{def5.4} there exist, for each $n$ and $\epsilon>0$, elements $\sigma,\tau\in N_T(x,\delta_n[x])$ such that $(\tau^{-1}\sigma x,x^\prime)\in\epsilon$. Moreover, if $F$ is any finite subset of $\langle T\rangle$ it can also be arranged that $(s\gamma x,sx)\in\delta_n\ \forall s\in F$, for $\gamma=\sigma$ and $\tau$.

Since $(T,X)$ is minimal, then $N_T(x,\delta[x])$ is syndetic in $T$ for all $\delta>0$. We can select a sequence $(\sigma_1,\tau_1), (\sigma_2,\tau_2), \dotsc$ inductively as follows. First we choose $\sigma_1,\tau_1\in N_T(x,\delta_1[x])$ and $F_0=\{e\}$ with $\rho(\tau_1^{-1}\sigma_1x,x^\prime)<\delta_1$. Having chosen $(\sigma_1,\tau_1), \dotsc, (\sigma_n,\tau_n)$ let $F_n$ be the finite set of elements of $\langle T\rangle$ which are representable as
$$
t=\tau_1^{\epsilon_1}\sigma_1^{\epsilon_1^\prime}\dotsm\tau_n^{\epsilon_n}\sigma_n^{\epsilon_n^\prime}\quad \textrm{where }\epsilon_i=0\textrm{ or }-1\textrm{ and } \epsilon_i^\prime=0\textrm{ or }1.
$$
Then choose $\sigma_{n+1},\tau_{n+1}\in N_T(x,\delta_{n+1}[x])$ in such a way that
\begin{enumerate}
\item[(a)] $\rho(s\gamma x,sx)<\delta_{n+1}\ \forall s\in F_n$, where $\gamma\in\{\sigma_{n+1},\tau_{n+1}\}$; and
\item[(b)] $\rho(\tau_{n+1}^{-1}\sigma_{n+1}x,x^\prime)<\delta_{n+1}$.
\end{enumerate}
Based on the sequence $\{(\sigma_n,\tau_n)\}$, we define $\alpha_1=\tau_1, \alpha_2=\sigma_1\tau_2$ in $T$, and in general,
$$
\alpha_n=\sigma_1\dotsm\sigma_{n-1}\tau_n\in T\quad (n=2,3,\dotsc).
$$
If $m<n$ we have from (a) that
\begin{equation*}
\rho(\alpha_mx,\alpha_nx)\le\sum_{j=0}^{n-m-1}\rho(\alpha_{m+j}x,\alpha_{m+j+1}x)
\le\sum_{j=0}^{n-m-1}(2\delta_{m+j}+\delta_{m+j+1})
\end{equation*}
tends to $0$ as $m\to\infty$. Thus $\lim_n\alpha_nx$ exists under $\rho$ because $\sum_{n=1}^{\infty}\delta_n<\infty$. Letting $y$ be the $\rho$-limit, we now claim $\lim_{m\to\infty}\alpha_m^{-1}y=x^\prime$ under $\rho$. To see this we note that if $n>m$, then
\begin{equation*}
\alpha_m^{-1}\alpha_n=\tau_m^{-1}\sigma_{m-1}^{-1}\dotsm\sigma_1^{-1}\sigma_1\dotsm\sigma_{n-1}\tau_n=\tau_m^{-1}\sigma_m\dotsm\sigma_{n-1}\tau_n
\end{equation*}
and therefore by argument as above, for $n>m$,
$$
\rho\left(\alpha_m^{-1}\alpha_nx,\tau_m^{-1}\sigma_mx\right)\le\sum_{k=m+1}^n\delta_k\to 0\quad \textrm{as }m\to\infty.
$$
Thus by (b),
\begin{equation*}
\lim_{m\to\infty}\rho\left(\alpha_m^{-1}y,x^\prime\right)=\lim_{m\to\infty}\lim_{n\to\infty}\rho\left(\alpha_m^{-1}\alpha_n x,x^\prime\right)=0
\end{equation*}
and $\alpha_m^{-1}y\to x^\prime$ under $\rho$. By choosing a subnet $\{\beta_i\}$ from the sequence $\{\alpha_n\}$ in $T$, there are points $z,x^{\prime\prime}\in X$ such that under the pointwise topology,
$\beta_ix\to z$, $\lim_j\lim_i\beta_j^{-1}\beta_ix=x^{\prime\prime}$;
and moreover, $x^{\prime\prime}(s)=x^\prime(s)\, \forall s\in S$.
Therefore we have concluded the statement (1) of Lemma~\ref{lem5.10}.

(2). Since $T$ is countable, so $G$ is also countable. Thus $X$ is naturally metrizable. Now use the metric of $X$ in place of $\rho$ in the above arguments, we can conclude the statement (2) of Lemma~\ref{lem5.10}.
\end{proof}

We now return to the general invertible semiflow case with arbitrary compact $\textrm{T}_2$ phase space $X$. Theorem~\ref{thm5.11} below implies $D[x]\subseteq Q[x]$ (cf.~(2) of Theorem~\ref{thm5.6}).

\begin{thm}[{cf.~\cite[Theorem~1.2]{V68} for $T$ an abelian group}]\label{thm5.11}
Let $(T,X)$ be minimal invertible. Then:
\begin{enumerate}
\item[$(1)$] $V[x]$ is a dense subset of $D[x]$ for each $x\in X$.
\item[$(2)$] If $X$ is metrizable, then $V[x]=D[x]$.
\end{enumerate}
\end{thm}

\begin{proof}
(1): First by (1) of Theorem~\ref{thm5.6}, $V[x]\subseteq D[x]$ for all $x\in X$. Next, let $x_0\in X$ be any fixed.
We denote by $\{F_\alpha\}$ the collection of $n$-tuples $(n=n_\alpha)$ of continuous functions on $X$. Each $F_\alpha$ may be considered a vector-valued function on $X$, and if we define $f_{\alpha,x_0}(s)=F_\alpha(sx_0)$ for all $s\in\langle T\rangle$, the closure, $X_\alpha$, of the $T$-translates of $f_{\alpha,x_0}$ enjoys the properties of the space $X$ of Lemma~\ref{lem5.10} such that $X_\alpha=\{f_{\alpha,x}\,|\,x\in X\}$ where $f_{\alpha,x}(s)=F_\alpha(sx)$ for all $s\in\langle T\rangle$ and $x\in X$.
Let $\pi_\alpha\colon X\rightarrow X_\alpha$ be the natural mapping $x\mapsto f_{\alpha,x}$. This is an epimorphism from $(T,X)$ onto $(T,X_\alpha)$, since $tf_{\alpha,x}=f_{\alpha,tx}$ for all $t\in T$ and $x\in X$. So $(T,X_\alpha)$ is minimal. We denote by $D_\alpha$ and $V_\alpha$ the sets which play in $X_\alpha$ the role of $D$ and $V$. It is obvious that $\pi_\alpha(V[x_0])=V_\alpha[\pi_\alpha(x_0)]$, and therefore by Lemma~\ref{lem5.10} if $V_0[x_0]\subseteq D[x_0]$ is the closure of $V[x_0]$, then $\pi_\alpha(V_0(x_0))=D_\alpha[\pi_\alpha(x_0)]$. Now $X$ is the inverse limit of the inverse system $\{X_\alpha\}$ (indexed by the set of finite subsets of $C(X)$ directed by inclusion), and since $\pi_\alpha(D[x_0])\subseteq D_\alpha[\pi_\alpha(x_0)]=\pi_\alpha(V_0[x_0])\subseteq\pi_\alpha(D[x_0])$, $D[x_0]$ is the inverse limit of the sets $D_\alpha[\pi_\alpha(x_0)]$. It follows that $V_0[x_0]=D[x_0]$. Thus (1) has been proved.

(2): Using the construction of the proof of Lemma~\ref{lem5.10} we can easily see that $D[x]\subseteq V[x]$ if $X$ is a compact metric space.
The proof of Theorem~\ref{thm5.11} is therefore complete.
\end{proof}

It should be mentioned that when $(T,X)$ is a \textit{flow admitting an invariant probability measure}, our Theorem~\ref{thm5.11} is actually \cite[Theorems~13 and 16]{AGN} by different approaches.
Moreover, when $T$ is countable, we can easily obtain the following result.

\begin{thm}\label{thm5.12}
Let $(T,X)$ be minimal invertible with $T$ a countable semigroup. Then $V[x]=D[x]$ for each $x\in X$.
\end{thm}

\begin{proof}
This follows easily by an argument similar to that of Theorem~\ref{thm5.11} using (2) of Lemma~\ref{lem5.10}. So we omit the details here.
\end{proof}

Let us consider an explicit example.

\begin{exa}[Ellis' ``two circle'' minimal set]
Let $Y_0=\{0\}\times\mathbb{T}$ and $Y_1=\{1\}\times\mathbb{T}$, where $\mathbb{T}$ is the unit circle regarded as the real numbers modulo $2\pi$. Let $X=Y_0\cup Y_1$. $X$ will be topologized by specifying an open-closed neighborhood basis for each point as follows. For $\varepsilon>0$ and $y\in\mathbb{T}$, let
\begin{equation*}\begin{split}
U_\varepsilon(0,y)&=\{0\}\times[y,y+\varepsilon)\cup\{1\}\times(y,y+\varepsilon),\\
U_\varepsilon(1,y)&=\{0\}\times(y-\varepsilon,y)\cup\{1\}\times(y-\varepsilon,y].
\end{split}\end{equation*}
Then provided $X$ with the topology such that $\{U_\varepsilon(0,y)\}_{0<\varepsilon<\pi}$ and $\{U_\varepsilon(1,y)\}_{0<\varepsilon<\pi}$ are open-closed neighborhood bases of each $(0,y)\in Y_0$ and $(1,y)\in Y_1$ respectively, $X$ is a compact $\textrm{T}_2$ space which is not second-countable and therefore non-metrizable.

Let $T=\mathbb{Z}$ and $T\times X\rightarrow X$ by $(n,x)\mapsto nx=(0,y+n)$ if $x=(0,y)$ and $(1,y+n)$ if $x=(1,y)$. Then $(T,X)$ is a minimal flow with phase group $\mathbb{Z}$ such that
$$P[x]=V[x]=D[x]=\{(0,y),(1,y)\}$$
for all $x=(i,y)\in X$, $i=0,1$. Therefore, $(T,X)$ is not an a.a. flow.
\end{exa}

\begin{que}
$\,$
\begin{enumerate}
\item It would be of interest to know whether or not $V[x]$ is dense in $Q[x]$ for all $x\in X$ for a minimal invertible semiflow $(T,X)$; cf.~\cite{AGN} for $(T,X)$ in the flows with invariant measures.

\item Particularly, it is not known to us if (1) of Theorem~\ref{thm5.9} is still true when $T$ is a non-abelian semigroup (e.g. $T$ is an almost right C-semigroup).
\end{enumerate}
\end{que}

\begin{thm}\label{thm5.15}
Let $(T,X)$ be minimal invertible, whose phase space $X$ is a compact metric space. If $P(T,X)=Q(T,X)$, then $V[x]=Q[x]$ for all $x\in X$.
\end{thm}

\begin{proof}
By hypothesis, $P[x]=Q[x]$ for all $x\in X$. Further by Corollary~\ref{cor5.3} and Theorem~\ref{thm5.6}, we see $V[x]=Q[x]$ for all $x\in X$.
\end{proof}

We may consult \cite{C63} for some equivalent conditions of ``$P(X)=Q(X)$'' in the flows. By Lemma~\ref{lem3.2}, $P(X)=Q(X)$ if $(T,X)$ and $(X,T)$ both are l.a.p. invertible semiflows.

\section*{Acknowledgments}
This project was supported by National Natural Science Foundation of China (Grant Nos. 11790274 and 11431012) and PAPD of Jiangsu Higher Education Institutions.

\end{document}